\documentclass[11pt,reqno]{amsart}

\usepackage{amsfonts, amssymb, amscd}
\usepackage{amsmath}
\usepackage{verbatim}
\usepackage{amssymb}
\usepackage{mathrsfs}
\usepackage{graphicx}
\usepackage{cite}
\usepackage[all]{xy}
\usepackage{tikz}
\usepackage[toc,page]{appendix}
\usepackage{tikz-cd}

\usepackage{graphicx}
\usepackage{float}
\usepackage[margin=1.5in]{geometry}

\newtheorem{theorem}{Theorem}[section]
\newtheorem{lemma}[theorem]{Lemma}
\newtheorem{proposition}[theorem]{Proposition}
\newtheorem{definition}[theorem]{Definition}

\newtheorem{corollary}[theorem]{Corollary}
\newtheorem{remark}[theorem]{Remark}

\numberwithin{equation}{section}
\begin{document}
\title{Bott vanishing for elliptic surfaces}

\begin{abstract}
We explore Bott Vanishing for elliptic surfaces over $\mathbb{P}^1$. We show that Bott Vanishing is singnificantly affected by the geometric properties that whether there exists certain type of singular fibers on the elliptic fibration such as cuspidal fibers. For ample line bundle on the surface with large self-intersection, these geometric properties give criteria for Bott Vanishing.
\end{abstract}

\author{Chengxi Wang}
\address{UCLA Mathematics Department,
Box 951555, Los Angeles, CA 90095-1555} \email{chwang@math.ucla.edu}

\maketitle


{\bf MSC classes}: 14F17 (primary), 14J27, 14C17, 14C20 (secondary)

\tableofcontents
\section{Introduction}\label{section1}
The cohomology spaces of several special vector bundles on a variety reveal the geometric and topological character of the variety. In particular, some geometric properties of the variety are relate to the vanishing of higher cohomology spaces which is useful when calculating the spaces of sections of the vector bundle. We say Bott Vanishing holds for a smooth projective variety $X$ if $$H^j(X, \Omega_X^i\otimes A)=0$$ for all ample line bundle $A$, all $i\geq 0$ and all $j>0$. Bott Vanishing is much broader than Kodaira vanishing which is the case when $i$ equals the dimension of $X$.

\smallskip

All smooth projective toric varieties are proved to satisfy Bott Vanishing in \cite{BC, BTLM, Mustata, Fujino}. For a Fano variety, Bott Vanishing fails if it is not rigid. Bott vanishing is implied by the liftability of the Frobenius morphism \cite{BTLM}, \cite[Proposition 7.1.4]{AWZ}. Then Bott vanishes fails for a lot of flag varieties not isomorphic to a product of projective spaces such as Grassmannins not isomorphic to projective plane and quadrics of dimension at lease $3$ \cite[section 4]{BTLM}. Thus there still exist rigid Fano variety not satisfying Bott Vanishing.

\smallskip

Totaro gives an answer in \cite{BTbott} for Achinger-Witaszek-Zdanowicz's question \cite[after Theorem 4]{AWZ} that Bott vanishing holds for the quintic del Pezzo surfaces which are non-toric rationally connected varieties. Sebasti\'an Torres generalizes this example by giving in \cite{STorres} that Bott vanishing holds for every stable GIT quotient of $(\mathbb{P}^1)^n$
by the action of $PGL_2$, over an algebraically closed field of characteristic zero.
In \cite{BTbott}, there are also investigations for irrationally connected varieties.
Bott vanishing fails for all $K3$ surfaces of degree less than $20$ or at least $24$ with Picard number $1$ \cite[Theorem 3.2]{BTbott}, which is obtained by using recent work of Ciliberto-Dedieu-Sernesi and Feyzbakhsh \cite{CDS, Feyzbakhsh}.

 \smallskip

 The geometric properties of a $K3$ surface $X$ leading to this failure is the existence of elliptic curves of low degree on $X$ or a (possibly singular) Fano $3-$fold containing the $X$ as a hyperplane section. Theorems $5.1$, $6.1$ and $6.2$ in \cite{BTbott} give geometric interpretation for this failure: the existence of certain special type of singular fibers of the elliptic fibration plays an important role in whether $H^1(X, \Omega^1_X\otimes A)$ is zero for an ample line bundle $A$ which is the key point for Bott Vanishing. These results inspire the research in this paper.

\smallskip

 In this paper, we explore Bott vanishing for elliptic surfaces $X$ with an elliptic fibration $\pi:X\rightarrow \mathbb{P}^1$. For elliptic surfaces, one cannot expect Bott vanishing to hold for all ample line bundles. But we want to determine as explicitly as possible which ample line bundles $A$ have $H^j(X, \Omega_X^i\otimes A)=0$ all $i\geq 0$ and all $j>0$ (Definition \ref{botvanipair}). The crucial issue is testing whether $H^1(X, \Omega^1_X\otimes A)$ is zero by Proposition \ref{pairbott}.

 \smallskip

 From the main results of this paper, Theorem \ref{Thr=1} and Theorem \ref{Thr>1}, we know the vanishing of $H^1(X, \Omega^1_X\otimes A)$ is determined by geometric properties of the fibration $\pi$, i.e, whether there exists certain type of singular fibers of $\pi$ such as cuspidal fibers.
In Theorem \ref{Thr=1} and Corollary \ref{cor>1A^2}, we prove when the self-intersection number $A^2$ is large enough, the geometric properties are actually sufficient and necessary condition for the vanishing of $H^1(X, \Omega^1_X\otimes A)$ which we can use as a criterion. In other words, we get geometric explanation for the vanishing of the cohomology spaces. The notable aspect of the paper is that our conclusions are about the elliptic surfaces over $\mathbb{P}^1$ with arbitrary $\beta$ and $r$, where $\beta=\chi(\mathcal{O}_{X})$ is the Euler characteristic of the structure sheaf of $X$ and $r$ is the intersection number $A\cdot E$ for a fiber $E$ of $\pi$.

\smallskip

Moreover, Section \ref{Ellis} contains some background for the elliptic surfaces. We also obtain several surjective maps between the space of sections of ample line bundles on the surfaces which is needed in the proof of main results in Section \ref{mainthb}. In section \ref{examples}, we construct several concrete elliptic surfaces over $\mathbb{P}^1$ in the case when $r=1,2,3,4$ and apply the main theorems to get criteria for Bott Vanishing for a pair $(X, A)$ (Definition \ref{botvanipair}).

\smallskip

\medskip

\noindent{\it Acknowledgements.} I would like to thank Burt Totaro for his guidance and helpful suggestions.

\section{Elliptic surfaces}\label{Ellis}
In this section, we give several conclusions about elliptic surfaces as preparations for proof of main theorems.

\smallskip

In this paper, a variety over a field $k$ is referred to as an irreducible reduced separated scheme of finite type over $k$. We take a curve to mean a variety of dimension $1$.
An elliptic surface is defined as follows.
\begin{definition}\label{defellip}
Let $C$ be a smooth complex projective curve and $X$ be a smooth complex projective surface. A elliptic fibration of $X$ over $C$ is an surjective morphism $\pi:X\rightarrow C$ such that
\begin{enumerate}
\item almost all fibers are smooth of genus $1$;
\item all fibers are reduced;
\item there is no $(-1)-$curve contained in a fiber.
\end{enumerate}
A elliptic surface over $C$ is defined to be a smooth complex projective surface $X$ with an elliptic
fibration $\pi:X\rightarrow C$.
\end{definition}

\begin{remark}
Condition $(2)$ is not a standard part of the definition of an elliptic surface, but we will assume it for this paper.
Condition $(3)$ in Definition \ref{defellip} means the elliptic fibration is minimal. A $(-1)-$curve (on a smooth projective surface) is a curve isomorphic to $\mathbb{P}^1$ with self-intersection $-1$. If there are $(-1)-$curves on a smooth projective surface, we can successively blow-down $(-1)-$curves to reduce the surface to a minimal model. An elliptic fibration may have some reducible fibers.
\end{remark}

We give descriptions for the singular fibers and the singular locus of an elliptic fibration in Remark \ref{degreefi} and Lemma \ref{degreeS}.

\begin{remark}\label{degreefi}
Let $\pi:X\rightarrow \mathbb{P}^1$ be an elliptic fibration and $S$ be the singular locus of $\pi$ which is a zero-dimensional scheme in $X$. From Kodaira's table of singular fibers \cite[$\mathrm{V.7}$]{BPV}, \cite[Corollary $5.2.3$]{CD}, the irreducible singular fiber of $\pi$ is of type $\mathrm{I_1}$ which is a nodal cubic or type $\mathrm{II}$ which is a cuspidal cubic; the reducible singular fiber is of type $\mathrm{III}$ which is two copies of $\mathbb{P}^1$ tangent at a poin, or $\mathrm{IV}$ which is three copies of $\mathbb{P}^1$ through a point, or $\mathrm{I_n}$ $(n\geq 2)$ which consists of a cycle of $n$ rational curves, meeting transversally (Figure \ref{type234I}). We list as follows in local analytic coordinates the polynomial defining $\pi$ and the degree of $S$ at each non-smooth point of different types:

\ \ \ \ \ \ \ \ \ \ \ \ \ \ \ \ \ \ $\pi$ is given by  \ \  \ \ \ \ \ \ \ \ \ \ \ \ \ \ $\mathrm{degree}(S)$ is
\begin{itemize}
  \item type $\mathrm{I_1}$:  \ $\pi(x,y)=x^2-y^2$, \ \ \ \ \ \ \ $dim_{\mathbb{C}}\mathbb{C}[x,y]/(2x,-2y)=1$;
  \item type $\mathrm{II}$:   \ $\pi(x,y)=x^2-y^3$, \ \ \ \ \ \ \ $dim_{\mathbb{C}}\mathbb{C}[x,y]/(2x, -3y^2)=2$;
  \item type $\mathrm{III}$: $\pi(x,y)=x(x-y^2)$,\ \ \ \ \ $dim_{\mathbb{C}}\mathbb{C}[x,y]/(2x-y^2, -2xy)=3$;
  \item type $\mathrm{IV}$: $\pi(x,y)=x(x^2-y^2)$,  \  \ \ $dim_{\mathbb{C}}\mathbb{C}[x,y]/(3x^2-y^2, -2xy)=4$;
  \item type $\mathrm{I_n}$: \  $\pi(x,y)=xy$,\ \ \ \ \ \ \ \ \ \ \ \ \ \ $dim_{\mathbb{C}}\mathbb{C}[x,y]/(x,y)=1$,
\end{itemize}
where $S$ is given by $\partial\pi/\partial x = \partial\pi/\partial y =0$ and contained in the fiber of $\pi$ at zero point since $\pi$ is quasi-homogeneous in these coordinates.
\begin{figure}[H]
  \includegraphics[width=0.99\textwidth]{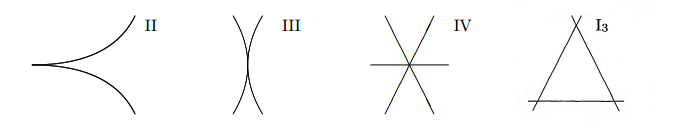}
  \caption{}
  \label{type234I}
\end{figure}
\end{remark}

The conditions $(2)$ and $(3)$ in Definition \ref{defellip} make it easy to compute the canonical bundle of an elliptic surface.

\begin{lemma}\label{Kx}
Let $\pi:X\rightarrow \mathbb{P}^1$ be an elliptic fibration. Then canonical bundle $$\mathcal{O}(K_X)=\pi^*(\mathcal{O}_{\mathbb{P}^1}(\chi(\mathcal{O}_{X})-2)),$$ where $\chi(\mathcal{O}_{X})$ is the Euler characteristic of the structure sheaf of $X$. In particular, if $\chi(\mathcal{O}_{X})>2$, then $X$ has Kodaira dimension $\kappa=1$.
\end{lemma}
\begin{proof}
By Definition \ref{defellip}, all fibers of $\pi$ are reduced. Hence $\pi$ has no multiple fibers. Then the canonical line bundle $\mathcal{O}(K_X)$ equals pull back of a line bundle on $\mathbb{P}^1$ with degree $\chi(\mathcal{O}(X))-2\chi(\mathcal{O}_{\mathbb{P}^1})$ \cite[Corollary 12.3]{BPV}. Thus $\mathcal{O}(K_X)=\pi^*(\mathcal{O}_{\mathbb{P}^1}(\chi(\mathcal{O}_{X})-2))$. If $\chi(\mathcal{O}_{X})>2$, the dimension of $$H^0(X, \mathcal{O}(nK_X))=H^0(\mathbb{P}^1, \mathcal{O}_{\mathbb{P}^1}(n(\chi(\mathcal{O}_{X})-2)))$$ goes to infinity as the integer $n$ increases to infinity.
\end{proof}

\begin{lemma}\label{degreeS}
Let $\pi:X\rightarrow \mathbb{P}^1$ be an elliptic fibration and $S$ be the non-smooth locus of $\pi$. Then $$\mathrm{degree}(S)=e(X)=12\chi(\mathcal{O}_{X}),$$ where $\chi(\mathcal{O}_{X})$ is the Euler characteristic of the structure sheaf of $X$ and $e(X)$ is the topological Euler characteristic of $X$.
\end{lemma}
\begin{proof}
By Theorem 6.10 in \cite{ScShi}, the Euler number $e(X)=\sum_{E}e(E)$, where $e(E)$ is Euler number for the fiber $E$ of $\pi$. The sum is finite since if $E$ is smooth, the Euler number $e(E)=b_0-b_1+b_2=1-2+1=0$, where $b_i$ are Betti numbers of $E$.
If $E$ is of type $\mathrm{I_n}$ $(n\geq1)$, then $e(E)=n$.
If $E$ is cuspidal cubic (type $\mathrm{II}$), then $e(E)=2$. Also $e(E)=3$ if $E$ is of type $\mathrm{III}$, $e(E)=4$ if $E$ is of type $\mathrm{IV}$. Let $n_1,n_2,n_3,n_4$ be the numbers of singular points on fibers of type $\mathrm{I_n}$ $(n\geq1)$, $\mathrm{II}$, $\mathrm{III}$ and $\mathrm{IV}$ respectively. Since there is one singular point on each fiber of type $\mathrm{II}$, $\mathrm{III}$ and $\mathrm{IV}$  and there are $n$ singular points on each fiber of type $\mathrm{I_n}$ $(n\geq1)$. Thus $e(X)=n_1+2n_2+3n_3+4n_4$. Noether's formula shows that $12\chi(\mathcal{O}_{X})=K_X^2+e(X)$. We know $K_X^2=0$ by Lemma \ref{Kx}. Therefore $$n_1+2n_2+3n_3+4n_4=e(X)=12\chi(\mathcal{O}_{X}).$$ By Remark \ref{degreefi}, $S$ has degree $1$ at each singular point on the fiber of type $\mathrm{I_n}$ $(n\geq1)$, degree $2$ at the singular point on the fiber of type $\mathrm{II}$, degree $3$ at the singular point on the fiber of type $\mathrm{III}$, degree $4$ at the singular point on the fiber of type $\mathrm{IV}$,  which implies $\mathrm{degree}(S)=n_1+2n_2+3n_3+4n_4=e(X)=12\chi(\mathcal{O}_{X})$.
\end{proof}

\begin{lemma}\label{beta>0}
Let $\pi:X\rightarrow \mathbb{P}^1$ be an elliptic fibration and $\chi(\mathcal{O}_{X})$ be the Euler characteristic of the structure sheaf of $X$. Then $\chi(\mathcal{O}_{X})\geq 0$. In particular, we have $\chi(\mathcal{O}_{X})\geq1$ unless the elliptic fibration $\pi:X\rightarrow \mathbb{P}^1$ is isomorphic to the trivial fibration $\mathbb{P}^1\times E\rightarrow \mathbb{P}^1$ for some elliptic curve $E$.
\end{lemma}
\begin{proof}
By Proposition $\mathrm{V}.12.2$ and the remark preceding it in \cite{BPV}, we have $\chi(\mathcal{O}_{X})\geq1$ unless all the smooth fibers of $\pi$ are isomorphic and singular fibers are of type $mI_0$ only.
By Definition \ref{defellip}, all fibers of $\pi$ are reduced. If singular fibers are of type $mI_0$ only, then all fibers are smooth. Thus when $\chi(\mathcal{O}_{X})=0$, all fibers of $\pi$ are isomorphic, i.e., $\pi$ is an elliptic fiber bundle. Theorem $\mathrm{V}.5.4$ in \cite{BPV} shows any elliptic fiber bundle over $\mathbb{P}^1$ is either a product or a Hopf surface. We know that any smooth projective variety has K\"ahler metric, thus its odd betti numbers have to be even numbers. Hopf surface is not projective since it has the first betti number $b_1=1$. Note that the elliptic surface $X$ we consider is projective (Definition \ref{defellip}).  Therefore, when $\chi(\mathcal{O}_{X})=0$, the elliptic fibration $\pi:X\rightarrow \mathbb{P}^1$ is isomorphic to the trivial fibration $\mathbb{P}^1\times E\rightarrow \mathbb{P}^1$ for some elliptic curve $E$. Thus the conclusion follows.

\end{proof}

We can obtain some information about the space of sections of a line bundle from its degree.

\begin{lemma}\label{h^0=r}
Let $F$ be a fiber of an elliptic fibration $\pi:X\rightarrow \mathbb{P}^1$. Suppose there is an ample line bundle
$L=\mathcal{O}(D)$ on $F$ with degree $r>0$. Then $h^0(F,L)=r$.
\end{lemma}
\begin{proof}
 We write $F=C_1+\cdots+C_t$, where $C_1,\ldots,C_t$ are irreducible components of $F$. Since $L$ is ample, thus $-D\cdot C_i<0$ for $i=1,\ldots,t$. Therefore $h^0(C_i,L^{\vee})=0$ for $i=1,\ldots,t$. Since $F$ is a fiber of $\pi$, thus $\omega_F$ is trivial and $F$ has arithmetic genus one. So $h^1(F,L)=h^0(F,L^{\vee}\otimes\omega_F)=h^0(F,L^{\vee})\leq h^0(C_1,L^{\vee})+\cdots+h^0(C_t,L^{\vee})=0$. Then by Riemann-Roch, we have $h^0(F,L)=\mathrm{degree}_{F}L+g-1=r$.
\end{proof}

\begin{lemma}\label{Tag}
Let $C$ be a curve with arithmetic genus $g$ and $L$ be an invertible $O_{C}-$module. Assume $Y\subset C$ is a nonempty $0-$dimensional closed subscheme. If $\mathrm{deg}(L)\geq 2g-1+\mathrm{deg}(Y)$, then $L$ is globally generated and $H^0(C, L)\rightarrow H^0(Y, L)$ is surjective.
\end{lemma}
\begin{proof}
See \cite[Tag 0E3D]{Stacks}.
\end{proof}

Let $F$ be a singular fiber of an elliptic fibration $\pi:X\rightarrow \mathbb{P}^1$ and $S_0$ be the singular subscheme of $F$ supported at the singular points of $F$. Assume $L$ is an ample line bundle on $F$ of degree $r$. There are results in \cite{BTbott} that the restriction $H^0(F, L)\rightarrow H^0(S_0, L)$ is surjective in some cases.
Lemma \ref{II}, \ref{III}, \ref{IV}, \ref{I_n} investigate the restriction map in several cases of certain type of singular fibers and value of $r$.

\begin{lemma}\label{II}
Let $F$ be a singular fiber of an elliptic fibration $\pi:X\rightarrow \mathbb{P}^1$ and $S_0$ be the singular subscheme of $F$ supported at the singular points of $F$. Assume $F$ is of type $\mathrm{II}$ and $L$ is an ample line bundle on $F$ of degree $r\geq 2$. Then the restriction $H^0(F, L)\rightarrow H^0(S_0, L)$ is surjective.
\end{lemma}
\begin{proof}
By computation in Remark \ref{degreefi}, the degree of $S_0$ is $2$.
Fiber $F$ is a curve with arithmetic genus $g=1$ since it is an irreducible and reduced fiber of the elliptic fibration. If $r\geq 3$, then $H^0(F, L)\rightarrow H^0(S_0, L)$ is surjective by Lemma \ref{Tag}.

\smallskip

If $r=2$, the line bundle is generated by global sections by Lemma \ref{Tag}, i.e., the line bundle is the base point free. So there is a global section $s_1$ of $L$ such that $s_1$ is nonzero at the singular point $p$. Let $T=\{s\in H^0(F, L)|s(p)=0\}$ be the subspace consists of sections of $L$ that vanish at $p$. Then Lemma \ref{h^0=r} shows $h^0(F, L)=r=2$. Hence space $T$ has dimension $2-1=1$. Let $s_2$ be a base element of $T$. Then the section $s_2$ vanishing at $p$ is not identically zero. Since the degree of $S_0$ is $2$, then $h^0(S_0, L)=2$. To show $H^0(F, L)\rightarrow H^0(S_0, L)$ is surjective, it is sufficient to prove the restrictions of the sections $s_1$ and $s_2$ to $S_0$ form a set of basis for $H^0(S_0, L)$. We only need to show the restriction $s_2$ to $S_0$ does not vanish on the whole scheme $S_0$ of degree $2$.

\smallskip

We consider the normalization $\sigma: \widetilde{F}\rightarrow F$ of $F$ with the unique point $q$ mapping to $p$, where $\widetilde{F}$ is isomorphic to $\mathbb{P}^1$.
By Remark \ref{degreefi}, locally the fiber $F$ is the curve defined by $x^2-y^3$ in $\mathbb{A}^2$, the singular locus $S_0$ is defined by ideal $\langle x, y^2\rangle$. Then $\sigma$ is given in local coordinates by $\sigma(t)=(t^3, t^2)$, and the point $q, p$ corresponed to $t=0$ and $(0,0)$ respectively. Then the pullback of sections $\sigma^{*}(x)$ and $\sigma^{*}(y^2)$ satisfying $\sigma^{*}(x)(t)=t^3$ and $\sigma^{*}(y^2)(t)=t^4$.

\smallskip

Assume that $s_2$ vanishes on $S_0$. Then locally $s_2\in \langle x, y^2\rangle$, which implies $\sigma^{*}(s_2)\in \langle \sigma^{*}(x), \sigma^{*}(y^2)\rangle$ locally. Therefore $\sigma^{*}(s_2)$ vanishes to order at least $3$ at $q$. Since $L$ is a line bundle of degree $r=2$ on $F$, the pullback $\sigma^{*}(L)$ is also a line bundle of degree $2$ on $\widetilde{F}$. A section of a line bundle of degree $2$ on $\widetilde{F}$ that vanishes at a point to degree at least $3$ must be identically zero, i.e, the section $\sigma^{*}(s_2)=0$. Hence the section $s_2=0$, which contradicts to that $s_2$ is a base element of $T$.

\end{proof}

\begin{lemma}\label{III}
Let $F$ be a singular fiber of an elliptic fibration $\pi:X\rightarrow \mathbb{P}^1$ and $S_0$ be the singular subscheme of $F$ supported at the singular points of $F$. Assume $F$ is of type $\mathrm{III}$ and $L$ is an ample line bundle on $F$ of degree $r\geq3$. Then the restriction $H^0(F, L)\rightarrow H^0(S_0, L)$ is surjective.
\end{lemma}

\begin{proof}
By remark \ref{degreefi}, in local coordinates near $p$, the fiber $F$ is defined by the equation $x(x-y^2)=0$ in $\mathbb{A}^2$ which is union of two components $C_1$ and $C_2$ isomorphic to $\mathbb{P}^1$. The two curves $C_1$ and $C_2$ are tangent at one point $p$ with intersection number $C_1\cdot C_2=2$ in the surface $X$. And $S_0$ supported at $p$ with degree $3$ is the closed subscheme of $\mathbb{A}^2$ defined by $x-\frac{1}{2}y^2=0, y^3=0$ locally.

\smallskip

Let the degree of ample line bundle $L$ on $C_i$ is $r_i$ for $i=1,2$ respectively, where $r_1, r_2$ are positive integers with $r_1+r_2=r\geq3$. Then one of $r_1, r_2$ is at least $2$. The space of sections $H^0(C_i, L)=H^0(\mathbb{P}^1, \mathcal{O}(r_i))$ for $i=1,2$. A section of $L$ on $F$ is given by a section of $\mathcal{O}(r_1)$ on $C_1$ and a section of $\mathcal{O}(r_1)$ on $C_2$ with value and first derivatives agree at $p$, that is, $$H^0(F, L)=\{(f,g)\in \oplus_{i=1}^2H^0(\mathbb{P}^1, \mathcal{O}(r_i)) \ | \ f(p)=g(p), f'(p)=g'(p)\},$$ where the restrictions $s|_{C_1}=f$ and $s|_{C_2}=g$.

\smallskip

Without loss of generality, we assume $r_2\geq 2$ since we can switch the names of $C_1$ and $C_2$ if necessary. Let $Y_2$ be the unique subscheme of degree $3$ in the smooth curve $C_2$ supported at $p$. Since $\mathrm{deg}_{C_2}(L)=r_2\geq 2$ and the genus of $C_2$ is zero, then $H^0(\mathbb{P}^1, \mathcal{O}(r_2))=H^0(C_2, L)\rightarrow H^0(Y_2, L)$ is surjective by Lemma \ref{Tag}. Hence
there exists a section $g$ of $\mathcal{O}(r_2)$ such that $g(p)$, $g'(p)$ and $g''(p)$ can be any value we want. Similarly, we apply Lemma \ref{Tag} to the unique subscheme of degree $2$ in $C_1$ supported at $p$. Then there exists a section $f$ of $\mathcal{O}(r_1)$ such that $f(p)$ and $f'(p)$ can be any value we want. Thus we can pick sections as follows.

\smallskip

Let $s_1$ be a section in $H^0(F, L)$ given by $f_1,g_1$ such that $$f_1(p)=g_1(p)\neq 0;$$ and $s_2$ be a section given by $f_2,g_2$ such that $$f_2(p)=g_2(p)=0   \text{ and } f_2'(p)=g_2'(p)\neq0;$$ and $s_3$ be a section given by $f_3,g_3$ such that $$f_3(p)=g_3(p)=0, f_3'(p)=g_3'(p)=0, f_3''(p)=0 \text{ and } g_3''(p)\neq 0,$$ where $(f_i, g_i)\in H^0(\mathbb{P}^1, \mathcal{O}(r_1))\oplus H^0(\mathbb{P}^1, \mathcal{O}(r_2))$ for $i=1,2,3$ and $f_3$ is the zero section.

\smallskip

Let $C_3$ be the curve given by $x-\frac{1}{2}y^2=0$ locally that contains $S_0$. Since $s_1(p)\neq 0$, $s_2(p)=0$ and the restriction of $s_2$ to $C_3$ has $s_2'(p)\neq 0$, the sections $s_1$ and $s_2$ is linear independent on $S_0$. It remains to show the restriction of $s_3$ to $C_3$ has $s_3''(p)\neq0$. If we show this, then $s_3$ is not identically zero on $S_0$. Also by $s_3(p)=0,s_3'(p)=0$, the sections $s_1, s_2, s_3$ are linear independent on $S_0$. Therefore the three sections restrict to a basis for the space of sections $H^0(S_0, L)=\mathbb{C}^3$, which complete the proof.

\smallskip

We choose local coordinates so that the curve $C_1$ is given by $x=0$ and $C_2$ is given by $x=y^2$. Since $s_3$ restricts to $0$ on $C_1$, then $s_3$ as a regular function on $\mathbb{A}^2$ has the form $$s_3=ax+bx^2+cxy  \ \   (\mathrm{mod} (x,y)^3).$$ Then the restriction of $s_3$ to $C_2$ is given by $$s_3|_{C_2}=ay^2+by^4+cy^3\equiv ay^2 \ \ (\mathrm{mod} y^3).$$ Since the second derivative of $s_3|_{C_2}$ equals $g_3''(p)\neq 0$, then $a\neq 0$. The restriction of $s_3$ to $C_3$ is given by
$$s_3|_{C_3}=\frac{a}{2}y^2+\frac{b}{2}y^4+\frac{c}{2}y^3\equiv \frac{a}{2}y^2 \ \ (\mathrm{mod} y^3).$$ Then the second derivative of $s_3|_{C_3}$ is $\frac{a}{2}\neq0$.

\end{proof}

\begin{lemma}\label{IV}
Let $F$ be a singular fiber of an elliptic fibration $\pi:X\rightarrow \mathbb{P}^1$ and $S_0$ be the singular subscheme of $F$ supported at the singular points of $F$. Assume $F$ is of type $\mathrm{IV}$ and $L$ is an ample line bundle on $F$ of degree $r\geq4$. Then the restriction $H^0(F, L)\rightarrow H^0(S_0, L)$ is surjective.
\end{lemma}
\begin{proof}
By remark \ref{degreefi}, in local coordinates near $p$, the fiber $F$ is defined by the equation $x(x-y)(x+y)=0$ in $\mathbb{A}^2$ which is union of three components $C_1$, $C_2$ and $C_3$ isomorphic to $\mathbb{P}^1$ intersecting at a point $p$. And $S_0$ supported at $p$ with degree $4$ is the closed subscheme of $\mathbb{A}^2$ defined by $x^2-\frac{1}{3}y^2=0, xy=0$ locally.

\smallskip

Let the degree of ample line bundle $L$ on $C_i$ is $r_i$ for $i=1,2,3$ respectively, where $r_1, r_2, r_2$ are positive integers with $r_1+r_2+r_3=r\geq4$. Then one of $r_1, r_2, r_3$ is not less than $2$.
The space of sections $H^0(C_i, L)=H^0(\mathbb{P}^1, \mathcal{O}(r_i))$ for $i=1,2,3$. A section $s\in H^0(F, L)$ is given by $(f,g,h)\in \oplus_{i=1}^3 H^0(\mathbb{P}^1, \mathcal{O}(r_i))$ such that $f(p)=g(p)=h(p)$ and $f'(p),g'(p),h'(p)$ satisfy a certain fixed linear equation, that is,  $e_1f'(p)+e_2g'(p)+e_3h'(p)=0$ for some fixed nonzero numbers $e_1, e_2, e_3$ depending on a choice of trivialization of the tangent spaces of $C_1, C_2, C_3$  at $p$. The restrictions $s|_{C_1}=f$, $s|_{C_2}=g$ and $s|_{C_3}=h$.

\smallskip

Without loss of generality, we assume $r_1\geq1,r_2\geq2$ and $r_3\geq1$ since we can switch the names of $C_1$, $C_2$ and $C_3$ if necessary.
As consequences of Lemma \ref{Tag} applied to the subschemes supported at $p$ of degree $2$ in $C_1$, of degree $3$ in $C_2$ and of degree $2$ in $C_3$ respectively, there exists a section $f$ of $\mathcal{O}(r_1)$ such that $f(p)$ and $f'(p)$ can be any value we want, and there exists a section $g$ of $\mathcal{O}(r_2)$ such that $g(p)$, $g'(p)$ and $g''(p)$ can be any value we want, and there is a section $h$ of $\mathcal{O}(r_1)$ such that $h(p)$ and $h'(p)$ can be any value we want. Thus we can pick sections as follows.

\smallskip

Let $s_1$ be a section in $H^0(F, L)$ given by $f_1,g_1,h_1$ such that $$f_1(p)=g_1(p)=h_1(p)\neq 0;$$ and $s_2$ be a section given by $f_2,g_2,h_2$ such that $$f_2(p)=g_2(p)=h_2(p)=0 \text{ and } f_2'(p)\neq0, g_2'(p)\neq0, h_2'(p)\neq0;$$ and $s_3$ be a section given by $f_3, g_3, h_3$ such that $$f_3(p)=g_3(p)=h_3(p)=0, f_3'(p)=0, g_3'(p)\neq0 \text{ and } h_3'(p)\neq0;$$ and $s_4$ be a section given by $f_4, g_4, h_4$ such $$f_4(p)=g_4(p)=h_4(p)=0, f_4'(p)=g_4'(p)=h_4'(p)=0 \text{ and } g_4''(p)\neq 0,$$
where $(f_i, g_i, h_i)\in \oplus_{i=1}^3 H^0(\mathbb{P}^1, \mathcal{O}(r_i))$, $e_1f_i'(p)+e_2g_i'(p)+e_3h_i'(p)=0$, and $f_4$ and $h_4$ are the zero sections.

\smallskip

We claim that the four sections $s_1, s_2, s_3, s_4$ restrict to a basis for the space of sections $H^0(S_0, L)=\mathbb{C}^4$. This completes the proof. Assume that $\sum_{i=1}^{4}d_is_i=0$. Since $s_1(p)\neq0$ and $s_i(p)=0$ for $i=2,3,4$, then $d_1=0$. Then by the derivative of the restriction of $s_2$ to $C_1$ is $f_2'(p)\neq 0$, and the derivative of the restriction of $s_i$ to $C_1$ is $f_i'(p)= 0$ for $i=3,4$, so $d_2=1$. Then by the second derivative of the restriction of $s_3$ to $C_2$ is $g_3'(p)\neq 0$, and the second derivative of the restriction of $s_4$ to $C_2$ is $g_4'(p)= 0$, so $d_3=0$. Hence $d_4s_4=0$ on $S_0$. It remains to show that $s_4$ is not identically zero on $S_0$.

\smallskip

We choose local coordinates so that the curve $C_1$ is given by $x=0$, $C_2$ is given by $x-y=0$ and $C_3$ is given by $x+y=0$. On a small neighborhood $U\subseteq C_1$ of $p$, the space of sections $H^0(U, L|_U)=\mathbb{C}[x,y]/\langle x\rangle$. Since $s_4$ restrict to zero on $C_1$,
then $s_4$ as a regular function on $\mathbb{A}^2$ has the form $$s_4=ax+bx^2+cxy  \ \   (\mathrm{mod} (x,y)^3).$$ The restriction of $s_4$ to $C_2$ is given by $$s_4|_{C_2}=ay+by^2+cy^2 \ \   (\mathrm{mod} y^3).$$ Since the derivative of $s_4|_{C_2}=g_4'(p)=0$, then $a=0$. Since the second derivative of $s_4|_{C_2}$ equals $g_4''(p)\neq0$, then $b+c\neq0$. The restriction of $s_4$ to $C_3$ is given by $$s_4|_{C_3}=bx^2-cx^2 \ \   (\mathrm{mod} x^3).$$ Since the second derivative of $s_4|_{C_3}$ equals $h_4''(p)=0$, hence $b=c$. Therefore $b\neq 0$. Since $S_0$ is given by $x^2-\frac{1}{3}y^2=0, xy=0$ locally, the corresponding coordinate ring $\mathbb{C}[x,y]/\langle x^2-\frac{1}{3}y^2, xy\rangle$ has basis $1, y, y^2, x$. The restriction of $s_4$ to $S_0$ is give by $s_4|_{S_0}=\frac{b}{3}y^2$. Since $b\neq 0$, the section $s_4$ is not identically zero on $S_0$.
\end{proof}

\begin{lemma}\label{I_n}
Let $F$ be a singular fiber of an elliptic fibration $\pi:X\rightarrow \mathbb{P}^1$ and $S_0$ be the singular subscheme of $F$ supported at the singular points of $F$. Assume $F$ is of type $I_n$ with $n\geq1$ and $L$ is an ample line bundle on $F$. Then the restriction $H^0(F, L)\rightarrow H^0(S_0, L)$ is surjective.
\end{lemma}
\begin{proof}
If $n=1$, fiber $F$ is a nodal cubic curve with a singular point $p$.
We consider the normalization $\widetilde{F}$ of $F$ with the normalization map $\sigma: \widetilde{F}\rightarrow F$, where $\widetilde{F}\cong \mathbb{P}^1$ is a smooth curve with genus $0$ and the inverse image $\sigma^{-1}(p)$ has two points $q_1, q_2$. We have $\mathrm{deg}_{F}(L)\geq1$ since $L$ is ample on $F$. Since $\sigma$ is normalization map, so $\mathrm{deg}_{\widetilde{F}}(\sigma^*L)=\mathrm{deg}_{F}(L)\geq1$. Since $\sigma^{-1}(p)$ is a $0-$dimensional closed subscheme of $\widetilde{F}$ of degree $2$, the restriction $H^0(\widetilde{F}, \sigma^*L)\rightarrow H^0(\sigma^{-1}(p), \sigma^*L)$ is surjective by Lemma \ref{Tag}. This implies that there are sections of $\sigma^*L$ on $\widetilde{F}$ taking any values at $q_1$ and $q_2$. The sections of $L$ on $F$ is the subspace of sections of $\sigma^*L$ on $\widetilde{F}$ that takes same value at $q_1$ and $q_2$, i.e., $$H^0(F, L)=\{s\in H^0(\widetilde{F}, \sigma^*L) \ | \ s(q_1)=s(q_2) \}.$$ Thus there is a section $s$ in $H^0(F, L)$ such that $s(p)\neq 0$ corresponding to a section in $H^0(\widetilde{F}, \sigma^*L)$ which takes same nonzero value at $q_1$ and $q_2$. Then $s$ restricts to a nonzero section in $H^0(S_0, L)$. The computations in Remark \ref{degreefi} show that the degree of $S_0$ is $1$. Thus $H^0(S_0, L)$ has dimension $1$, which implies $H^0(F, L)\rightarrow H^0(S_0, L)$ is surjective.

\smallskip

We now consider the case of $n\geq2$. By Remark \ref{degreefi}, the fiber $F$ consists of a cycle of $n$ curves $C_0,\ldots, C_{n-1}$ isomorphic to $\mathbb{P}^1$. The curves $C_i$ and $C_{i+1}$ meeting transversely at a point $Q_{i,i+1}$ for $i=0,\ldots, n-1$. We consider the index $i$ modulo $n$. The singular locus $S_0$ supported at $n$ points $Q_{0,1}, Q_{1,2}, \ldots, Q_{n-1,0}$ has degree $1$ at each point.

\smallskip

Let the positive integer $r_i$ be the degree of ample line bundle $L$ on $C_i$ for $i=0,\ldots, n-1$ respectively.
The space of sections $H^0(C_i, L)=H^0(\mathbb{P}^1, \mathcal{O}(r_i))$ for $i=0,\ldots, n-1$.
A section $s\in H^0(F, L)$ is given by $(f_i)_{i=1}^{n}\in \oplus_{i=0}^{n-1} H^0(\mathbb{P}^1, \mathcal{O}(r_i))$ such that $f_i(Q_{i,i+1})=f_{i+1}(Q_{i,i+1})$ for $i=0,\ldots, n-1$,
where the restrictions $s|_{C_i}=f_i$ for $i=0,\ldots, n-1$.
The space of sections $$H^0(S_0, L)=\oplus_{i=1}^{n}H^0(Q_{i,i+1}, L)=\mathbb{C}^{\oplus n}$$ has basis $e_i=(0,\ldots, 0, 1, 0,\ldots,0)$ for $ i=0,\ldots, n-1$, where $i-$th entry equals $1$ and the other entries equal $0$.

\smallskip

 For any two points on $\mathbb{P}^1$, there exists an automorphism mapping the two points to $[1:0]$ and $[0:1]$ respectively. The space of sections $H^0(C_i, L)=H^0(\mathbb{P}^1, \mathcal{O}(r_i))$ has elements $x^{r_i}, y^{r_i}$ and zero, where $r_i\geq 1$. The section $x^{r_i}$ maps $[1:0]$ and $[0:1]$ to $1, 0$ respectively. The section $y^{r_i}$ maps $[1:0]$ and $[0:1]$ to $0, 1$ respectively. The zero section maps both $[1:0]$ and $[0:1]$ to $0$. Thus there are three sections in $H^0(C_i, L)$: one maps $Q_{i-1,i}$ and $Q_{i,i+1}$ to $1, 0$ respectively, one maps $Q_{i-1,i}$ and $Q_{i,i+1}$ to $0, 1$ respectively, and one maps both $Q_{i-1,i}$ and $Q_{i,i+1}$ to $0$.

\smallskip

Let $f_i$ be the section in $ H^0(C_i, L)$ such that $f(Q_{i-1,i})=0$ and $f(Q_{i,i+1})=1$, $f_{i+i}$ be the section in $ H^0(C_{i+1}, L)$ such that $f(Q_{i,i+1})=1$ and $f(Q_{i+1,i+2})=0$. For all $j\neq i, i+1$, let $f_j$ be the section in $ H^0(C_j, L)$ such that $f(Q_{j-1,j})=f(Q_{j,j+1})=0$. Then the section $s_i\in H^0(F, L)$ given by $f_1, f_2, \ldots,f_n$ restricts to the base element $e_i$ in $H^0(S_0, L)$.  Therefore the restriction $H^0(F, L)\rightarrow H^0(S_0, L)$ is surjective.

\end{proof}

\section{Bott vanishing for elliptic surfaces}\label{mainthb}
In this section, we focus on proving main results: Theorem \ref{Thr=1}, Theorem \ref{Thr>1}, Corollary \ref{cor>1} and Corollary \ref{cor>1A^2}. Let $X$ be a smooth complex projective surface with an elliptic fibration $\pi:X\rightarrow \mathbb{P}^1$ and ample line bundle $A$ on $X$, $\beta=\chi(\mathcal{O}_{X})$ be the Euler characteristic of the structure sheaf of $X$ and $E$ be a fiber of $\pi$. Proposition \ref{pairbott} tells us that the vanishing of the cohomology space $H^1(X, \Omega^1_{X}\otimes A)$ plays a key role in testing the Bott Vanishing for ample line bundles $A$ on these elliptic surfaces. So our main results focus on investigating whether $H^1(X, \Omega^1_{X}\otimes A)$ vanishes.

\smallskip

Let $r$ be the intersection number $A\cdot E$.
When $r=1$, Theorem \ref{Thr=1} shows the cohomology space $H^1(X, \Omega^1_{X}\otimes A)$ vanishes if and only if $\pi$ has no fibers of type $\mathrm{II}$ and the self-intersection $A^2$ is larger than a number relate to $\beta$.
When $r\geq1$ and $A-\beta E$ is nef and big, Theorem \ref{Thr>1} states that $H^1(X, \Omega^1_{X}\otimes A)$ vanishes imples $\pi$ has no fibers of certain types and the converse holds with an additional condition about the dimension of a space of sections of the line bundle $A-(11\beta-1)E$. Then Remark \ref{L-Kx} shows that $A-(12\beta-2)E$ is nef and big is slightly stronger than the additional condition. Thus under the assumption $A-(12\beta-2)E$ is nef and big, Corollary \ref{cor>1} gives a criterion of whether $H^1(X, \Omega^1_{X}\otimes A)$ vanishes, i.e, $H^1(X, \Omega^1_{X}\otimes A)$ vanishes if and only if $\pi$ has no fibers of certain types. Moreover, computing the self-intersection $A^2$ may be easier than checking $A-(12\beta-2)E$ is nef and big. When $r\geq2$, under the assumption that the self-intersection $A^2$ is larger than a number relate to $\beta$ and $r$, Corollary \ref{cor>1A^2} displays a criterion of whether $H^1(X, \Omega^1_{X}\otimes A)$ vanishes, i.e, $H^1(X, \Omega^1_{X}\otimes A)$ vanishes if and only if $\pi$ has no fibers of certain types.

\smallskip

Now we give a formula for the Euler characteristic $\chi(X,\Omega_{X}^1\otimes A)$ for any elliptic fibration $X\rightarrow \mathbb{P}^1$.

\begin{proposition}\label{Euleromega}
Let $X$ be a smooth complex projective surface with an elliptic fibration $\pi:X\rightarrow \mathbb{P}^1$, $\beta=\chi(\mathcal{O}_{X})$ be the Euler characteristic of the structure sheaf of $X$, and $A$ be an ample line bundle on $X$. Then
$$\chi(X,\Omega_{X}^1\otimes A)=A^2-10\beta.$$
\end{proposition}
\begin{proof}
By Riemann-Roch and $c_1(\Omega_{X}^1)=c_1(K_X)$, $c_1(T_X)=-c_1(K_X)$, where $K_X$ is the canonical bundle of $X$, we get
\begin{align*}
\chi(X,\Omega_{X}^1\otimes A)&=\int_{X}td(T_X)ch(\Omega_{X}^1\otimes A)=\int_{X}td(T_X)ch(\Omega_{X}^1)ch(A)
\\&=\int_{X}\big(1+\frac{1}{2}c_1(T_X)+\frac{c_1^2(T_X)+c_2(T_X)}{12}\big)
\\&\cdot\big(2+c_1(\Omega_{X}^1)+\frac{c_1^2(\Omega_{X}^1)-2c_2(\Omega_{X}^1)}{2}\big)\cdot\big(1+c_1(A)+\frac{c_1^2(A)}{2}\big)
\\&=\int_Xc_1^2(A)+\frac{1}{6}\int_Xc_1^2(K_X)-\frac{5}{6}\int_Xc_2(T_X)
\\&=A^2+\frac{1}{6}K^2_X-\frac{5}{6}\int_Xc_2(T_X).
\end{align*}
For elliptic surface $X$, we have $K_X^2=0$. Then by Noether's formula $12\chi(\mathcal{O}_{X})=K_X^2+e(X)$, we get $\int_Xc_2(T_X)=e(X)=12\beta$.
Therefore $\chi(X,\Omega_{X}^1\otimes A)=A^2-10\beta$.
\end{proof}

\begin{definition}\label{botvanipair}
Let $X$ be a smooth projective variety and $A$ be an ample line bundle on $X$. We say that $(X, A)$ satisfies Bott vanishing if $H^j(X, \Omega^{i}\otimes A)=0$ for all $i\geq 0$ and all $j>0$.
\end{definition}

\begin{remark}\label{Euleromega1}
We have $\chi(X,\Omega_{X}^1\otimes A)=h^0(X,\Omega_{X}^1\otimes A)-h^1(X,\Omega_{X}^1\otimes A)+h^2(X,\Omega_{X}^1\otimes A).$ By Proposition \ref{Euleromega}, if $A^2<10\beta$, then $\chi(X,\Omega_{X}^1\otimes A)$ is negative. This implies $h^1(X,\Omega_{X}^1\otimes A)$ is nonzero. Therefore Bott vanishing fails. If $A^2\geq10\beta$, then $\chi(X,\Omega_{X}^1\otimes A)$ is nonnegative. So there are still possibilities that Bott vanishing holds for $(X, A)$.
\end{remark}

The following Proposition tells us whether $H^1(X, \Omega^1_{X}\otimes A)$ vanishes is crucial for Bott Vanishing for $(X, A)$.

\begin{proposition}\label{pairbott}
Let X be a smooth complex projective surface with an elliptic fibration $\pi:X\rightarrow \mathbb{P}^1$ and $E$ be a fiber of $\pi$ and $\beta=\chi(\mathcal{O}_{X})$ be the Euler characteristic of the structure sheaf of $X$.
Assume $A$ is an ample line bundle on $X$ such that $A-(\beta-2)E$ is nef and big. Then $(X, A)$ satisfies Bott vanishing if and only if $H^1(X, \Omega^1_{X}\otimes A)=0$.
\end{proposition}
\begin{proof}
Kodaira-Akizuki-Nakano vanishing theorem (\cite[Theorem 4.2.3]{Lazarsfeld}) shows $H^j(X,\Omega^i\otimes L)=0$ for all ample line bundles $L$ on $X$ and all $i+j>\dim(X)$. Thus $H^j(X,\Omega^2\otimes L)=0$ for $j>0$ and $H^j(X,\Omega^1\otimes L)=0$ for $j\geq2$.

\smallskip

 Since $A-K_X=A-(\beta-2)E$ is nef and big, by Kawamata-Viehweg vanishing, the cohomology spaces $H^j(X, A)= 0$ for $j>0$. It remains to check whether $H^1(X,\Omega_{X}^1\otimes A)=0$. Then we get the conclusion.
\end{proof}

The singular locus of an elliptic fibration plays an important role in investigating the cohomology spaces.

\begin{proposition}\label{Simpose}
Let X be a smooth complex projective surface with an elliptic fibration $\pi:X\rightarrow \mathbb{P}^1$ and $E$ be a fiber of $\pi$ and $\beta=\chi(\mathcal{O}_{X})$ be the Euler characteristic of the structure sheaf of $X$.
Assume $A$ is an ample line bundle on $X$ such that $A-\beta E$ is nef and big. Then the cohomology space $H^1(X, \Omega^1_{X}\otimes A)=0$ is equivalent to that the non-smooth locus $S$ (as a scheme) of $\pi$ imposes linearly independent conditions on sections in $H^0(X,A+\beta E)$.
\end{proposition}

\begin{proof}
We consider the relative K\"ahler differentials $\Omega^1_{X/\mathbb{P}^1}$ and the exact sequence of coherent sheaves on $X$:
\begin{equation}\label{exrKdi}
0\rightarrow \pi^*\Omega^1_{\mathbb{P}^1}\rightarrow \Omega^1_X\rightarrow \Omega^1_{X/\mathbb{P}^1}\rightarrow 0,
\end{equation}
where $\pi^*\Omega^1_{\mathbb{P}^1}=\pi^*\mathcal{O}(-2)=\mathcal{O}(-2E)$. We tensor the exact sequence with $A$ to get the induced long exact sequence of cohomology spaces:
$$H^1(X,A-2E)\rightarrow H^1(X,\Omega^1_X\otimes A)\rightarrow H^1(X,\Omega^1_{X/\mathbb{P}^1}\otimes A)\rightarrow H^2(X,A-2E).$$
By Lemma \ref{Kx}, the canonical bundle $K_X=(\beta-2)E$. Then $A-2E-K_X=A-\beta E$ is nef and big. Thus both these two cohomology spaces $H^1(X,A-2E)$ and $H^2(X,A-2E)$ are zero by Kawamata-Viehweg vanishing. Therefore the map from $H^1(X,\Omega^1_X\otimes A)$ to $H^1(X,\Omega^1_{X/\mathbb{P}^1}\otimes A)$ is isomorphic.

\smallskip

The relative dualizing sheaf $\omega_{X/\mathbb{P}^1}=\omega_X\otimes (\pi^*\omega_{\mathbb{P}^1})^*$ is a line bundle which describes the difference between the canonical line bundles of $X$ and $\mathbb{P}^1$. The following exact sequence expresses that the difference from $\Omega^1_{X/\mathbb{P}^1}$ to a line bundle is on the singular locus $S$ (as a scheme) of the elliptic fibration:
\begin{equation}\label{exdiffonS}
0\rightarrow \Omega^1_{X/\mathbb{P}^1}\rightarrow \omega_{X/\mathbb{P}^1}\rightarrow \omega_{X/\mathbb{P}^1}|_S \rightarrow 0.
\end{equation}
By restricting to open subset $U=X-S$, all the coherent sheaves in the exact sequence (\ref{exrKdi}) are vector bundles. Then we take determinants to get $\mathrm{det}(\Omega^1_X)=\mathrm{det}(\pi^*\Omega^1_{\mathbb{P}^1})\otimes \Omega^1_{X/\mathbb{P}^1}$ on $U$, which implies $\Omega^1_{X/\mathbb{P}^1}=\omega_X\otimes\mathcal{O}(2E)$ on $U$. Thus $\omega_{X/\mathbb{P}^1}=\omega_X\otimes\mathcal{O}(2E)$ since $\omega_{X/\mathbb{P}^1}$ is a line bundle on the whole $X$. Since $K_X=(\beta-2)E$,
the exact sequence (\ref{exdiffonS}) tensored with $A$ induces a long exact sequence of cohomology spaces:
$$H^0(X,A+\beta E)\xrightarrow {\theta} H^0(S,A+\beta E)\rightarrow H^1(X,\Omega^1_X\otimes A)\rightarrow H^1(X,A+\beta E).$$
The line bundle $A+\beta E-K_X=A+2E$ is ample because $A$ is ample and $E$ is nef.
Thus $H^1(X,A+\beta E)=0$ by Kodaira vanishing. As a result, the cohomology space $H^1(X,\Omega^1_X\otimes A)$ is zero if and only if the restriction map $\theta$ is surjective, which means that $S$ imposes independent conditions on sections of $A+\beta E$.
\end{proof}

\begin{theorem}\label{Thr=1}
Let X be a smooth complex projective surface with an elliptic fibration $\pi:X\rightarrow \mathbb{P}^1$ such that all fibers are reduced, and $\beta=\chi(\mathcal{O}_{X})$ is the Euler characteristic of the structure sheaf of $X$.
Assume $A$ is an ample line bundle on $X$ with $A\cdot E=1$, where $E$ be a fiber of $\pi$. Then $H^1(X, \Omega^1_{X}\otimes A)\neq0$ if and only if $A^2\leq 21 \beta - 3$ or $\pi$ has a fiber of type $\mathrm{II}$.
\end{theorem}

Before the proof of Theorem \ref{Thr=1}, we analyze the elliptic fibration and the ample line bundle given in the theorem.

\begin{lemma}\label{lemmar=1}
With the same assumptions and notations in Theorem \ref{Thr=1}, we conclude that any fiber of $\pi$ is irreducible and has multiplicity one. When $A^2\geq 3\beta$, the ample line bundle $A$ is linear equivalent to $A_0+mE$ for a section $A_0$ of $\pi$ and some integers $m\geq 2\beta$, satisfying $A_0^2=-\beta$, $A^2=-\beta+2m$ and $K_X\cdot A_0=\beta-2$.
\end{lemma}

\begin{proof}
Since $\mathcal{O}(E)=\pi^*\mathcal{O}_{\mathbb{P}^1}(1)$, any fiber of $\pi$ is an effective divisor linearly equivalent to $E$. By the assumption that $A.E=1$, the intersection number of $A$ with any fiber of $\pi$ equals one.  Also since $A$ is ample, any fiber of $\pi$ is irreducible and has multiplicity one.
By Lemma \ref{Kx}, the canonical bundle $K_X=(\beta-2)E$.

\smallskip

Riemann-Roch shows
\begin{align*}
\beta &=h^0(X, A)-h^1(X, A)+h^2(X, A)
\\&=\frac{A\cdot(A-K_X)}{2}+\beta=\frac{A^2}{2}-\frac{\beta-2}{2}+\beta=\frac{A^2}{2}+\frac{\beta}{2}+1.
\end{align*}
Since $A^2\geq 3\beta$, the intersection number $(K_X-A).A=(\beta-2)E.A-A^2\leq\beta-2-3\beta=-2\beta-2$. Lemma \ref{beta>0} shows $\beta\geq 0$. Thus $(K_X-A).A<0$ which implies $K_X-A$ is not effective. Hence $h^2(X, A)=h^0(X, K_X-A)=0$. Then $h^0(X, A)\geq \frac{A^2}{2}+\frac{\beta}{2}+1>0.$  Therefore $A$ is linearly equivalent to an effective divisor.

\smallskip

We can write $A=\sum_ia_iC_i$, where $a_i>0$ and $C_i$ are irreducible curves on  $X$. Since $A\cdot E=1$ and $E$ is nef, only one irreducible curve $C_{i_0}$ in the sum has $C_{i_0}\cdot E=1$ and all the other $C_i$ with $i\neq i_0$ have $C_i\cdot E=0$. Then those $C_i$ with $i\neq i_0$ are curves supported in fibers and  $C_{i_0}$ is a section of $\pi$ which we denoted it by $A_0$. Since every fiber is irreducible and linearly equivalent to $E$, then $A$ is linear equivalent to $A_0+mE$, where $A_0$ is isomorphic to $\mathbb{P}^1$.
Therefore $\mathrm{deg}_{A_0}(K_{A_0})=2g-2=-2$. Since $K_X\cdot A_0=(\beta-2)E.A_0=\beta-2$, adjunction formula shows $A_0^2=\mathrm{deg}_{A_0}(K_{A_0})-K_X\cdot A_0=-2-(\beta-2)=-\beta$. Then $A^2=-\beta+2m$. The assumption $A^2\geq 3\beta$ implies $m\geq2\beta$.

\end{proof}

\begin{proof}[{\bf Proof of Theorem \ref{Thr=1}.}] If $A^2< 3\beta$, then $H^1(X, \Omega^1_{X}\otimes A)\neq0$ by Remark \ref{Euleromega1}. Now we assume $A^2\geq 3\beta$.
By Lemma \ref{lemmar=1}, we have $A-\beta E=A_0+(m-\beta)E$.
It can be checked that the line bundle $A-\beta E$ is nef and big when $A^2\geq3 \beta$. Indeed, if $C$ is an irreducible curve different from $A_0$ on $X$, then $A_0\cdot C+(m-\beta)E\cdot C\geq 0$ when $m\geq \beta$; if $C=A_0$, then $(A_0+(m-\beta)E)\cdot A_0=-\beta+(m-\beta)\geq0$ when $m\geq 2\beta$. We also have $(A_0+(m-\beta)E)^2=-\beta+2(m-\beta)>0$ when $m>\frac{3}{2}\beta$. The assumption that $A^2\geq3 \beta$ implies $m\geq 2 \beta$. And Lemma \ref{beta>0} shows $\beta\geq 0$. Thus all the inequalities about $m$ above hold. Then by Proposition \ref{Simpose}, the cohomology space $H^1(X, \Omega^1_{X}\otimes A)=0$ if and only if the non-smooth locus $S$ (as a scheme) imposes linearly independent conditions on sections in $H^0(X,A+\beta E)$.

\smallskip

By Lemma \ref{Kx}, the canonical bundle $K_X=(\beta-2)E$.
Because $A$ is ample and $E$ is nef, the line bundle $A+\beta E-K_X=A+2E$ is ample. Then $h^1(X,A+\beta E)=h^2(X,A+\beta E)=0$. Riemann-Roch shows $$h^0(X,A+\beta E)=\frac{(A+\beta E)^2}{2}-\frac{(A+\beta E)\cdot K_X}{2}+\chi(\mathcal{O}(X)).$$ It is implied by Lemma \ref{lemmar=1} that $(A+\beta E)^2=(A_0+(m+\beta)E)^2=-\beta+2(m+\beta)=2m+\beta$ and $(A+\beta E)\cdot K_X=A_0\cdot K_X+(m+\beta)E\cdot K_X=\beta-2$. Then $$h^0(X,A+\beta E)=\frac{2m+\beta}{2}-\frac{\beta-2}{2}+\beta=m+1+\beta.$$

\smallskip

We consider the map $\phi:H^0(\mathbb{P}^1, \mathcal{O}(\beta+m))\rightarrow H^0(X,A_0+(\beta+m)E)$ given by pulling back a section of $\mathcal{O}(\beta+m)$ on $\mathbb{P}^1$ by $\pi$ to $H^0(X, \pi^*\mathcal{O}(\beta+m))=H^0(X, (\beta+m)E)$ and then multiplying by the canonical section of the line bundle $\mathcal{O}(A_0)$ on $X$. The map $\phi$ is injective and the section in the image of $\phi$ has zero set in $X$ of the form $A_0+E_1+\cdots+E_{\beta+m}$ for some fibers $E_1,\ldots,E_{\beta+m}$ of $\pi$. Since both $H^0(\mathbb{P}^1, \mathcal{O}(\beta+m))$ and $H^0(X,A_0+(\beta+m)E)$ have same dimension $\beta+m+1$, thus $\phi$ is also surjective.
Therefore, the linear system of $A+\beta E=A_0+(\beta+m)E$ is the set of divisors $A_0+E_1+\cdots+E_{\beta+m}$ for some fibers $E_1,\ldots,E_{\beta+m}$ of $\pi$. By Lemma \ref{lemmar=1}, all fibers of $\pi$ are irreducible. Thus $\pi$ can only have singular fibers of type $\mathrm{I}_1$ and type $\mathrm{II}$.

\smallskip

If $\pi$ only have singular fibers of type $\mathrm{I}_1$, the number of singular points equals $\mathrm{degree}(S)=12\beta$ by Lemma \ref{degreeS}. Then $S$ consists of $12\beta$ nodal points $p_1,\ldots,P_{12\beta}$ in $12\beta$ different fibers $F_1,\ldots, F_{12\beta}$ respectively. We know $H^0(S, A+\beta E)$ has a basis $\{s_i\}_{i=1}^{12\beta}$ such that $s_i(p_i)\neq0$ and $s_i(p_j)=0$ for all $j\neq i$. Thus $S$ impose linearly independent conditions on sections of $A+\beta E$ on $X$ if and only if there are sections $f_i\in H^0(X,A+\beta E)$ such that $f_i(p_i)\neq0$ and $f_i(p_j)=0$ for all $j\neq i$.

\smallskip

On one hand, if $m+\beta\geq 12\beta-1$, then for any $i\in \{1,\ldots, 12\beta\}$,  we can pick a section $f\in H^0(X,A+\beta E)$ with zero locus $$(f)_0=A_0+F_1+\ldots+F_{i-1}+F_{i+1}+\ldots+F_{12\beta}+E_{1}\ldots+E_{m+\beta-(12\beta-1)},$$ where $E_j$ is a fiber different from $F_i$ for all $j=1,\ldots, m+\beta-(12\beta-1)$. Since a section of elliptic fibration never passes through singular points of $\pi$, i.e., $A_0\cap S=\emptyset$, we have $f(p_i)\neq0$ and $f(p_j)=0$ for all $j\neq i$. Hence we can let $f_i=f$.

\smallskip

On the other hand, if there is a section $f_i\in H^0(X,A+\beta E)$ such that $f_i(p_i)\neq0$ and $f_i(p_j)=0$ for all $j\neq i$, i.e., $(f_i)_0$ contains $12\beta-1$ points $p_j$ for all $j\neq i$. Since $(f_i)_0=A_0+E_1+\cdots+E_{\beta+m}$ for some fibers $E_1,\ldots,E_{\beta+m}$ of $\pi$ and $A_0\cap S=\emptyset$, we get $m+\beta\geq 12\beta-1$. Therefore $S$ impose linearly independent conditions on sections of $A+\beta E$ if and only if $m+1+\beta\geq 12\beta$, that is $A^2=-\beta+2m\geq21\beta-2$.  Therefore, when $\pi$ has no singular fibers of type $\mathrm{II}$, the condition $A^2\leq 21 \beta - 3$ is equivalent to $H^1(X, \Omega^1_{X}\otimes A)\neq0$.

\smallskip

Now we consider the case that $\pi$ has a singular fiber $E_0$ of type $\mathrm{II}$.
Let $p$ be the cusp point in $E_0$ and $S_0=E_0\cap S$ be the component of $S$ supported at $p$. By Remark \ref{degreefi}, the degree of $S_0$ is $2$. Thus $H^0(S_0, A+\beta E)$ has dimension two with basis $h_1, h_2$ such that $h_1(p)\neq0$, $h_2(p)=0$ and $h_2$ is not identically zero on $S_0$.

\smallskip

Let $g$ be a nonzero section in $H^0(X,A+\beta E)$ vanishing at $p$. The vanishing locus of $g$ is a divisor $A_0+E_1+\cdots+E_{\beta+m}$ for some fibers $E_1,\ldots,E_{\beta+m}$ of $\pi$.
The cusp point $p$ is contained in this divisor. Since $A_0\cap S=\emptyset$, there is some $i\in\{1,\ldots,\beta+m\}$ such that $E_i$ contain $p$. Thus $E_0=E_i$, which implies $g$ vanishes on the whole fiber $E_0$. Thus $g$ is identically zero on $S_0$. This implies there is no section in $H^0(X,A+\beta E)$ restricting to $h_2$. Therefore $S$ does not impose linearly independent conditions on sections of $A+\beta E$, that is, $H^1(X, \Omega^1_{X}\otimes A)\neq0$.
\end{proof}

\begin{theorem}\label{Thr>1}
Let X be a smooth complex projective surface with an elliptic fibration $\pi:X\rightarrow \mathbb{P}^1$ such that all fibers are reduced, $\beta=\chi(\mathcal{O}_{X})$ be the Euler characteristic of the structure sheaf of $X$, $E$ be a fiber of $\pi$, and $A$ be an ample line bundle on $X$ with $r:= A\cdot E$. Assume $A-\beta E$ is nef and big. If $\pi$ has a fiber of type $\mathrm{II}$ when $r=1$, a fiber of type $\mathrm{III}$ when $r=2$, or a fiber of type $\mathrm{IV}$ when $r=3$, then $H^1(X, \Omega^1_{X}\otimes A)\neq0$. The converse holds with additional condition that $h^0(L)-h^0(L-E)=r$, where $L=A-(11\beta-1)E$.
\end{theorem}

Before the proof of Theorem \ref{Thr>1}, we first see two corollaries.

\begin{remark}\label{L-Kx}
In Theorem \ref{Thr>1}, the converse holds if we make the additional condition a slightly stronger, that is, the converse holds if $L-(\beta-1)E$ is nef and big. Indeed, by Lemma \ref{Kx}, the canonical bundle $K_X=(\beta-2)E$. Since $L-E-K_X=L-(\beta-1)E$ are nef and big, the line bundle $L-K_X=L-(\beta-1)E+E$ is also nef and big. By Kawamata-Viehweg vanishing, the higher cohomology spaces of $L$ and $L-E$ are all zero. The by Riemann-Roch, we have
$$h^0(X, L)=\frac{L^2}{2}-\frac{L\cdot K_X}{2}+\beta \text{ and } h^0(X, L-E)=\frac{(L-E)^2}{2}-\frac{(L-E)\cdot K_X}{2}+\beta.$$ Since $L^2=A^2-2(11\beta-1)r$ and $(L-E)^2=A^2-2(11\beta)r$, we get
$$h^0(X, L)-h^0(X, L-E)=\frac{L^2}{2}-\frac{(L-E)^2}{2}=r.$$ Therefore the condition that $L-(\beta-1)E$ is nef and big implies $h^0(L)-h^0(L-E)=r$. So the converse holds. Moreover, we get the following corollary.
\end{remark}

\begin{corollary}\label{cor>1}
Let X be a smooth complex projective surface with an elliptic fibration $\pi:X\rightarrow \mathbb{P}^1$ such that all fibers are reduced, $\beta=\chi(\mathcal{O}_{X})$ be the Euler characteristic of the structure sheaf of $X$, $E$ be a fiber of $\pi$, and $A$ be an ample line bundle on $X$ with $r:= A\cdot E$. Assume $A-(12\beta-2)E$ is nef and big. Then $H^1(X, \Omega^1_{X}\otimes A)\neq0$ if and only if $\pi$ has a fiber of type $\mathrm{II}$ when $r=1$, a fiber of type $\mathrm{III}$ when $r=2$, or a fiber of type $\mathrm{IV}$ when $r=3$.
\end{corollary}
\begin{proof}
If $\beta=0$, then $A-\beta E=A$ is ample, thus is also nef and big. If $\beta\geq 1$ implies $11 \beta-2\geq0$.
Since $E$ is nef and $A-(12\beta-2)E$ is nef and big, the line bundle $A-\beta E=A-(12\beta-2)E+(11 \beta-2)E$ is also nef and big. Let $L=A-(11\beta-1)E$, we have $A-(12\beta-2)E=L-(\beta-1)E$. Then the conclusion follows from Remark \ref{L-Kx} and Theorem \ref{Thr>1}.
\end{proof}

\begin{lemma}
Let X be a smooth complex projective surface with an elliptic fibration $\pi:X\rightarrow \mathbb{P}^1$ such that all fibers are reduced, $\beta=\chi(\mathcal{O}_{X})$ be the Euler characteristic of the structure sheaf of $X$, $E$ be a fiber of $\pi$, and $A$ be an ample line bundle on $X$ with $r:= A\cdot E\geq2$. If $A^2\geq 2r^2\beta+25r\beta-4r-2\beta$, then $A-(12\beta-2)E$ is nef and big.
\end{lemma}

\begin{proof}
If $\beta=0$, then $A-(12\beta-2)E=A+2E$ is nef and big since $A$ is ample and $E$ is nef. Now we assume $\beta\geq1$.
Let $k_0=12\beta-2+r\beta$. Then $A^2\geq 2r^2\beta+25r\beta-4r-2\beta=2rk_0+(r-2)\beta$.
We first show the line bundle $A-k_0E$ is liner equivalent to an effective divisor. Since $(A-k_0E)^2=A^2-2k_0r$, $(A-k_0E)\cdot K_X=(\beta-2)r$ and the assumption $A^2\geq 2rk_0+(r-2)\beta$,  Riemann-Roch shows
\begin{align*}
\chi(X, A-k_0E) &=\frac{(A-k_0E)^2}{2}-\frac{(A-k_0E)\cdot K_X}{2}+\beta\\
&=\frac{A^2}{2}-k_0r-\frac{r}{2}(\beta-2)+\beta\\
&\geq rk_0+\frac{r-2}{2}\beta-k_0r-\frac{r}{2}\beta+r+\beta=r\geq2.
\end{align*}
We also have
\begin{align*}
(K_X-A+k_0E)\cdot A&=r(\beta-2)-A^2+k_0r\\
&\leq r(\beta-2)-2rk_0-(r-2)\beta+k_0r=-2r+2\beta-rk_0\\
&=-2r+2\beta-r(12\beta-2+r\beta)=(2-12r-r^2)\beta<0
\end{align*} since $\beta\geq1$ and $r\geq2$.
Thus $K_X-A+k_0E$ is not effective. Hence $h^2(X, A-k_0E)=h^0(X, K_X-A+k_0E)=0$.  Therefore $h^0(X, A-k_0E)>0$ which imples $A-k_0E$ is linearly equivalent to an effective divisor which we denote by $D$. We write $$D=\sum_{i=1}^na_iD_i=\sum_{i=1}^la_iD_i+\sum_{i=l+1}^na_iD_i$$ for some $1\leq l\leq n$, where $a_i\geq1$, $D_i$ are different irreducible curves on $X$ and $D_i$ is contained in a fiber of $\pi$ if and only if $i\in \{l+1,\ldots,n\}$.

\smallskip

We claim that $D+(r\beta)E$ is nef by checking $(D+(r\beta)E)\cdot C$ for any irreducible curve $C$ on $X$. If $C\neq D_i$ for all $i=1,\ldots,n$, then $(D+(r\beta)E)\cdot C=\sum_{i=1}^na_i(D_i\cdot C)+(r\beta)E\cdot C\geq0$ since $E$ is nef. Note that $E\cdot D_i=0$ for $i\in \{l+1,\ldots,n\}$ since $D_i$ is contained in a fiber of $\pi$. Hence if $C=D_i$ for some $i\in \{l+1,\ldots,n\}$, we get $$(D+(r\beta)E)\cdot C=(D+k_0E)\cdot D_i=A\cdot D_i>0$$ since $A$ is ample.
Now we assume $C=D_j$ for some $j\in \{1,\ldots,l\}$. Let $b_i=D_i\cdot E$. Since $D\cdot E=A\cdot E=r $ and $E\cdot D_i=0$ for all $i\in \{l+1,\ldots,n\}$, thus $E\cdot\sum_{i=1}^la_iD_i=r$. Hence $1\leq b_i\leq r$ and $1\leq a_i\leq r$ for $i\in \{1,\ldots,l\}$. We have $$(D+(r\beta)E)\cdot D_j=a_jD_j^2+\sum_{i\neq j}a_iD_i\cdot D_j+ (r\beta)E\cdot D_j\geq a_jD_j^2+(r\beta)b_j.$$ By adjunction formula $K_{D_j}=(K_X+D_j)|_{D_j}$,
we get
$$2g-2=\mathrm{deg}(K_{D_j})=(K_X+D_j)\cdot D_j=(\beta-2)E\cdot D_j +D_j^2,$$ where $g$ is the arithmetic genus of $D_j$. Thus $D_j^2=-(\beta-2)b_j+2g-2\geq -(\beta-2)b_j-2$. If $\beta=1$, then $$(D+(r\beta)E)\cdot D_j\geq a_j(b_j-2)+rb_j\geq a_j(1-2)+r=r-a_j\geq0 $$ since $ b_j\geq1$ and $1\leq a_j\leq r$. If $\beta\geq 2$, then $(\beta-2)b_j+2>0$. Therefore
\begin{align*}
(D+(r\beta)E)\cdot D_j&\geq -a_j((\beta-2)b_j+2)+(r\beta)b_j\\
&\geq -r((\beta-2)b_j+2)+(r\beta)b_j=2rb_j-2r\geq 0.
\end{align*} since $ b_j\geq1$ and $a_j\leq r$.

\smallskip

Moreover, we have
\begin{align*}
(D+(r\beta)E)^2&=(A-k_0E+(r\beta)E)^2=A^2+2(r\beta-k_0)E\cdot A\\
&\geq 2rk_0+(r-2)\beta+2(r\beta-k_0)r=(r-2)\beta+2r^2\beta>0
\end{align*}since $\beta\geq1$ and $r\geq2$.
Thus $A-(12\beta-2)E=A-k_0E+(r\beta)E=D+(r\beta)E$ is nef and big.
\end{proof}

This Lemma allows us to check that $A-(12\beta-2)E$ is nef and big by computing the self intersection number $A^2$ which might be much easier. Then by Corollary \ref{cor>1}, we get the following corollary.

\begin{corollary}\label{cor>1A^2}
Let X be a smooth complex projective surface with an elliptic fibration $\pi:X\rightarrow \mathbb{P}^1$ such that all fibers are reduced, $\beta=\chi(\mathcal{O}_{X})$ be the Euler characteristic of the structure sheaf of $X$, $E$ be a fiber of $\pi$, and $A$ be an ample line bundle on $X$ with $r:= A\cdot E\geq2$. Assume $A^2\geq 2r^2\beta+25r\beta-4r-2\beta$. Then $H^1(X, \Omega^1_{X}\otimes A)\neq0$ if and only if $\pi$ has a fiber of type $\mathrm{III}$ when $r=2$, or a fiber of type $\mathrm{IV}$ when $r=3$.
\end{corollary}

As preparation for the proof of Theorem \ref{Thr>1}, the following lemma analyzes the restriction map of space of sections of line bundle from $X$ to singular locus.

\begin{lemma}\label{XrestriontoS}
Let $\pi:X\rightarrow \mathbb{P}^1$ be an elliptic fibration with all fibers reduced, $\beta=\chi(\mathcal{O}_{X})$ be the Euler characteristic of the structure sheaf of $X$ and $S$ be the singular locus of $\pi$. Assume for every singular fiber $E_0$ of $\pi$, the restriction $H^0(X,L)\rightarrow H^0(S_0, L)$ is surjective, where $S_0=S\cap E_0$. Then the restriction $\theta: H^0(X, L+(12\beta-1)E)\rightarrow H^0(S, L+(12\beta-1)E)$ is also surjective, where $E$ is a fiber of $\pi$.
\end{lemma}
\begin{proof}
Lemma \ref{degreeS} shows $\mathrm{degree}(S)=12\beta$. This implies the number of singular fibers is no more than $12\beta$. Let $E_1, E_2 \ldots, E_m$ be all the singular fibers of $\pi$ and $S_i=S\cap E_i$, where $m\leq 12\beta$. Then $$H^0(S, L+(12\beta-1)E)=\oplus_{i=1}^mH^0(S_i, L+(12\beta-1)E).$$ A base element in $H^0(S_i, L+(12\beta-1)E)$ corresponds to a section of $L+(12\beta-1)E$ on $S$ which is nonzero at a singular point in one singular fiber $E_i$ and zero at all singular points in other singular fibers. Thus any section of $L+(12\beta-1)E$ on $S$ can be write as a sum of the such sections. In order to show $\theta$ is surjective, it is sufficient to show such a section has a preimage in $H^0(X, L+(12\beta-1)E)$.

\smallskip

We arbitrarily pick a singular fiber from $E_1, E_2 \ldots, E_m$. Without loss of generality, we assume the singular fiber is $E_1$. Let   $E_2+\cdots+E_m+E_{m+1}+E_{12\beta-1}$ be an element in the linear system $|(12\beta-1)E|$, where all fibers $E_i$, $i=2,\ldots, 12\beta-1$ are different from $E_1$. Then there is a section $g\in H^0(X, (12\beta-1)E)$ such that the zero locus $(g)_0=E_2+\cdots+E_{12\beta-1}$. In other words, section $g$ vanishing on $E_i,i=2,\ldots,12\beta-1$ has no zero point on $E_1$.

\smallskip

Let $p$ be a singular point on $E_1$ and $s$ be a section of $L+(12\beta-1)E$ on $S$ which is nonzero at $p$ and zero at other singular points in $E_2,\ldots, E_{12\beta-1}$. We pick a base element $f\in H^0(S_1, L)$ such that $f(p)\neq 0$. By assumption, the restriction $H^0(X,L)\rightarrow H^0(S_1, L)$ is surjective. Then there is a section $\overline{f}\in H^0(X,L)$ such that $\overline{f}|_{S_1}=f$. Hence $\overline{f}(p)=f(p)\neq 0$. Now we get a section $g\otimes \overline{f}$ of $L+(12\beta-1)E$ on $X$ which is nonzero at $p$ and vanishes at other singular points in singular fibers $E_2,\ldots, E_{12\beta-1}$. Thus $s$ equals the restriction of $\frac{s(p)}{g(p)\overline{f}(p)}g\otimes \overline{f}$ to $S$.

\end{proof}

\begin{proof}[{\bf Proof of Theorem \ref{Thr>1}.}]

First, we show $H^1(X,\Omega_{X}^1\otimes A)\neq 0$ under the assumption that $A-\beta E$ is nef and big, and $\pi$ has a fiber of type $\mathrm{II}$ with $r=1$, a fiber of type $\mathrm{III}$ with $r=2$ or a fiber of type $\mathrm{IV}$ with $r=3$.

\smallskip

Let $E_0$ be a given fiber of type $\mathrm{II}$ when $r=1$, of type $\mathrm{III}$ when $r=2$, or of type $\mathrm{IV}$ when $r=3$.
Let $S_0$ be the connected component of $S$ supported at the singular point on $E_0$. By Remark \ref{degreefi}, the degree of $S_0$ is $2$ if the singular point is on a fiber of type $\mathrm{II}$ of $\pi$, the degree of $S_0$ is $3$ if the singular point is on a fiber of type $\mathrm{III}$ of $\pi$ and the degree of $S_0$ is $4$ if the singular point is on a fiber of type $\mathrm{IV}$. Thus $\mathrm{degree}(S_0)=r+1$. Hence $h^0(S_0,A+\beta E)=r+1$.

\smallskip

The line bundle $A+\beta E$ is ample since $A$ is ample and $E$ is nef.
We have $(A+\beta E)\cdot E_0=A\cdot E_0+\beta E\cdot E_0=r$, that is $A+\beta E$ has degree $r$ on the given fiber $E_0$. By Lemma \ref{h^0=r}, we have $h^0(E_0,A+\beta E)=r<h^0(S_0,A+\beta E)$. Therefore $H^0(X,A+\beta E)$ cannot map onto $H^0(S_0,A+\beta E)$ by restriction, which implies $S$ does not impose independent conditions on sections of $A+\beta E$. Since $A-\beta E$ is nef and big, the cohomology space $H^1(X,\Omega_{X}^1\otimes A)$ is not zero by Proposition \ref{Simpose}.

\smallskip

Conversely, we aim to show $H^1(X,\Omega_{X}^1\otimes A)=0$ under the assumption that $A-\beta E$ is nef and big, $h^0(L)-h^0(L-E)=r$, $\pi$ has no fiber of type $\mathrm{II}$ when $r=1$, $\pi$ has no fiber of type $\mathrm{III}$ when $r=2$ and $\pi$ has no fiber of type $\mathrm{IV}$ when $r=3$. By Proposition \ref{Simpose}, it is equivalent to show $S$ imposes linearly independent conditions on sections of $A+\beta E=L+(12\beta-1)E$.

\smallskip

Arbitrarily choose a fiber $E_0$ of $\pi$.
The short exact sequence $0\rightarrow\mathcal{O}(L-E_0)\rightarrow\mathcal{O}_X(L)\rightarrow\mathcal{O}_{E_0}(L)\rightarrow0$ induce long exact sequence
$$0\rightarrow H^0(X,L-E_0)\rightarrow H^0(X,L)\rightarrow H^0(E_0,L)\rightarrow\cdots.$$ By assumption that $h^0(L)-h^0(L-E)=r$, the restriction $H^0(X,L)\rightarrow H^0(E_0,L)$ has image of dimension $r$. Since $\mathcal{O}(E)$ restricted to $E_0$ is trivial, the line bundle $L$ restricted to $E_0$ equals $A$ restricted to $E_0$. Thus $L$ is ample on $E_0$. Then by Lemma \ref{h^0=r}, the dimension of $H^0(E_0,L)$ is $r$. Thus $H^0(X,L)$ restrict onto $H^0(E_0,L)$ for any fiber $E_0$ of $\pi$.

\smallskip

Let $E_0$ be a singular fiber of the elliptic fibration $\pi$ and $S_0=S\cap E_0$.
Since $\mathrm{deg}_{E_0}L=r$, number of irreducible components of $E_0$ is no more that $r$. Then $E_0$ can be of type $\mathrm{II}$ only when $r\geq2$; of type $\mathrm{III}$ only when $r\geq3$; of type $\mathrm{IV}$ only when $r\geq4$ and of type $I_n$ with $n\geq1$ only when $r\geq n$.
By Lemma \ref{II}, \ref{III}, \ref{IV}, \ref{I_n}, the restriction map from $H^0(E_0,L)$ to $H^0(S_0,L)$ is surjective for all these possible cases. Thus $H^0(X,L)$ maps onto $H^0(S_0,L)$ by restriction. Then by Lemma \ref{XrestriontoS}, the restriction from $H^0(X,L+(12\beta-1)E)$ to $H^0(S,L+(12\beta-1)E)$ is surjective, that is $S$ imposes linearly independent conditions on sections of $A+\beta E=L+(12\beta-1)E$. Therefore $H^1(X,\Omega_{X}^1\otimes A)=0$ by Proposition \ref{Simpose}.
\end{proof}

\section{Examples of elliptic fibrations}\label{examples}
In this section, Let $X$ be a smooth complex projective surface with an elliptic fibration $\pi:X\rightarrow \mathbb{P}^1$ and ample line bundle $A$ on $X$. Let $r$ be the intersection number $A\cdot E$ , where $E$ is a fiber of $\pi$. In this section, we explore elliptic surfaces in the cases when $r=1,2,3,4$ and apply Theorem \ref{Thr=1} and Theorem \ref{Thr>1}.

\subsection{$A\cdot E=1$}\label{e.g.r=1}
{\bf Elliptic fibrations with $r=A\cdot E=1$ for an ample line bundle $A$ on the elliptic surfaces, where $E$ is a fiber.}

\smallskip

Let $X$ be a smooth complex projective surface with an elliptic fibration $\pi:X\rightarrow \mathbb{P}^1$ and $A$ be an ample line bundle on $X$ with $A\cdot E=1$ for a fiber $E$ of $\pi$. By Lemma \ref{lemmar=1}, the ample line bundle $A$ is linear equivalent to $A_0+mE$ for a section $A_0$ of $\pi$ and some integers, satisfying $A_0^2=-\beta$, where $\beta$ be the Euler characteristic of the structure sheaf of $X$.
If $m\geq\frac{11}{2}\beta$, then $A^2=-\beta+2m\geq10\beta$. Remark \ref{Euleromega1}, So there are possibilities that Bott vanishing holds for $(X, A)$. Then Theorem \ref{Thr=1} can give a criterion for when the Bott vanishing holds for $(X, A)$ if $m$ is big enough.

\begin{proposition}\label{r=1iff}
Let $X$ be a smooth complex projective surface with an elliptic fibration $\pi:X\rightarrow \mathbb{P}^1$ and an ample line bundle $A=A_0+mE$, where $A_0$ is a section of $\pi$ and $E$ is a fiber of $\pi$. Let $\beta$ be the Euler characteristic of the structure sheaf of $X$. When $m> 11\beta-1$, Bott vanishing holds for $(X, A)$ if and only if $\pi$ has no fibers of type $\mathrm{II}$.
\end{proposition}
\begin{proof}


By Lemma \ref{Kx}, the canonical bundle $K_X=(\beta-2)E$. We consider the line bundle $A-K_X=A-\beta E+2E=A_0+(m-\beta+2)E$. By Lemma \ref{beta>0}, we know $\beta\geq0$. If $\beta=0$, then $A-K_X=A+2E$ is ample since $A$ is ample and $E$ is nef, thus is nef and big. If $\beta\geq 1$, then $A^2=A_0^2+2m=2m-\beta>2(11\beta-1)-\beta=21\beta-2>3 \beta$. It has been checked in the first paragraph in the proof of Theorem \ref{Thr=1} that line bundle $A-\beta E$ is nef and big when $A^2\geq3 \beta$. Thus $A-K_X=(A-\beta E)+2E$ is nef since $E$ is nef. Also $(A-K_X)^2=A_0^2+2(m-\beta+2)=2m-3\beta+4>2(11\beta-1)-3\beta+4=(22-3)\beta+2>0$, then  $A-K_X$ is big.
By Proposition \ref{pairbott}, we only need to check $H^1(X,\Omega_{X}^1\otimes A)$.
When $m> 11\beta-1$, we have $A^2>21\beta-3$.
By Theorem \ref{Thr=1}, the cohomology space $H^1(X,\Omega_{X}^1\otimes A)= 0$ if and only if $\pi$ has no fibers of type $\mathrm{II}$.
\end{proof}

An elliptic surface $X$ over $\mathbb{P}^1$ with a section can be described by a Weierstrass data $$(\mathbb{L}, \lambda, \mu)$$ over $\mathbb{P}^1$, where $\mathbb{L}=\mathcal{O}_{\mathbb{P}^1}(\beta)$ and $(\lambda, \mu)$ are global sections of $\mathbb{L}^4\oplus \mathbb{L}^6$ with discriminant $\Delta=4\lambda^3+27\mu^2$ not identically $0$ (see \cite[II-IV]{Miranda}), such that locally the fiber over $t$ in $\mathbb{P}^1$ is given by the Weierstrass equation $$y^2=x^3+\lambda(t)x+\mu(t).$$ The polynomials $\lambda$ and $\mu$ uniquely determine an elliptic surface $X$ over $\mathbb{P}^1$ with a section.

\smallskip

\cite[III.3]{Miranda} shows that Weierstrass fibrations over $\mathbb{P}^1$ in minimal form is one-to-one correspondence with the smooth minimal elliptic surfaces over $\mathbb{P}^1$ with section (possibly with non-reduced fibers).
Let $\nu_{p}(s)$ be the order of vanishing of a section $s$ of line bundle at a point $p$.
If for every $p\in \mathbb{P}^1$, either $\nu_{p}(\lambda)\leq3$ or $\nu_{p}(\mu)\leq5$, then Weierstrass data $(\mathbb{L}, \lambda, \mu)$ over $\mathbb{P}^1$ is in minimal form, thus the corresponding elliptic surface is minimal.
Moreover, the properties of $X$ are determined by $\lambda$ and $\mu$ in a slightly complicated way.
Table in \cite[IV.3.1]{Miranda} explains how to read off the types of singular fibers of elliptic fibration $\pi$ from the polynomials $\lambda$ and $\mu$.

\smallskip

Assume
$\beta\geq 1$.
And assume the elliptic fibration $\pi:X\rightarrow \mathbb{P}^1$ with section $A_0$ is determined by $\lambda\in H^0(\mathbb{P}^1, \mathcal{O}_{\mathbb{P}^1}(4\beta))$ and $\mu\in H^0(\mathbb{P}^1, \mathcal{O}_{\mathbb{P}^1}(6\beta))$ such that $\lambda=t^{4\beta}$ and $\mu=t^{6\beta}+t$ in an affine open set of $\mathbb{P}^1$. Since $\mu$ is nonzero at most points of $\mathbb{P}^1$ and the order of vanishing of $\mu$ at each zero point is one, we have $\nu_{p}(\mu)\leq 5$ for every $p\in \mathbb{P}^1$. Thus the elliptic surface $X$ is minimal. Locally, we have $$\Delta=4t^{12\beta}+27(t^{12\beta}+t^2+2t^{6\beta+1}).$$  Let $a, b, \delta$ be the order of vanishing of $\lambda$, $\mu$ and $\Delta$ at $t=0$ respectively. Then we have $\delta=2$, $a=4\beta\geq 1$ and $b=1$. Using the table in \cite[IV.3.1]{Miranda}, the elliptic fibration $\pi$ has a fiber of type $\mathrm{II}$ over $t=0$. By Proposition \ref{r=1iff}, Bott vanishing fails for $(X, A)$, where $A=A_0+mE$ with $m>11\beta-1$.

\subsection{$A\cdot E=2$}\label{e.g.r=2}
{\bf Elliptic fibrations with $r=A\cdot E=2$ for an ample line bundle $A$ on the elliptic surfaces, where $E$ is a fiber.}

\smallskip

Let $X$ be a smooth complex projective surface with a double cover $$\eta:X\rightarrow Z=\mathbf{P}_1\times \mathbf{P}_2$$ ramified over a smooth curve $B$ of bidegree $(2l,4)$ in $Z$ for some positive integer $l$, where $\mathbf{P}_1= \mathbb{P}^1= \mathbf{P}_2$. Let $p_1:Z\rightarrow \mathbf{P}_1$ and $p_2:Z\rightarrow \mathbf{P}_2$ be the two projections, $\pi=p_1\circ\eta$ and $f=p_2\circ\eta$. For a fiber $E$ of $\pi$, the morphism $f|_E: E\rightarrow \mathbf{P}_2$ obtained by restricting $f$ to $E$ is a double cover ramified over four points. Then $f|_E$ has degree $2$ and the ramification divisor $R_E$ of $f|_E$ has degree $4$. By Hurwitz's theorem (\cite[IV. Corollary 2.4]{Hartshorne}), we have $$2g(E)-2=\mathrm{deg}(f|_E)(2g(\mathbb{P}^1)-2)+\mathrm{deg}(R_E)=0,$$ which implies the fiber $E$ has arithmetic genus $g(E)=1$. Thus $\pi$ is an elliptic fibration.

\smallskip

By Leray spectral sequence, the cohomology spaces $H^i(X, \mathcal{F})=H^i(Z, \eta_*\mathcal{F})$ for a coherent sheaf $\mathcal{F}$ on $X$.
Since $\eta_*\mathcal{O}_X=\mathcal{O}_Z\oplus L$ (\cite[Lemma 17.1]{BPV}), where $L=\mathcal{O}_{Z}(-l,-2)$, we have
\begin{equation}\label{tensorL}
H^i(X, \eta^*\mathcal{G})=H^i(Z, \eta_*\eta^*\mathcal{G})=H^i(Z, \mathcal{G}\oplus(\mathcal{G}\otimes L))
\end{equation}
for a coherent sheaf $\mathcal{G}$ on $X$. Since $\eta$ is double cover and $\mathcal{O}_Z(B)=\mathcal{O}_Z(2l,4)$, the canonical bundle (\cite[Lemma 17.2]{BPV}) $$\omega_X=\eta^*\omega_Z\otimes\frac{1}{2}\eta^*\mathcal{O}_Z(B)=\eta^*(\mathcal{O}_Z(-2,-2)\otimes\mathcal{O}_Z(l,2))=\mathcal{O}_X(-2+l,0)=\pi^*(\mathcal{O}_{\mathbb{P}^1}(l-2)).$$
Then cohomology space $H^i(X, \omega_X)$ equals $$H^i(Z, \mathcal{O}_Z(-2+l,0)\oplus(\mathcal{O}_Z(-2+l,0)\otimes L))=H^i(Z, \mathcal{O}_Z(-2+l,0))\oplus H^i(Z, \mathcal{O}_Z(-2,-2)).$$
By K\"unneth formula, since $l\geq1$, we get $h^0(\mathcal{O}_Z(-2+l,0))=-2+l+1$, $h^1(\mathcal{O}_Z(-2+l,0))=0$ and $h^i(\mathcal{O}_Z(-2,-2))=0$ for $i=1,2$. Then $h^0(X, \omega_X)=-2+l+1$ and $h^1(X, \omega_X)=0$.
Let $\beta=\chi(\mathcal{O}_X)$ be the Euler characteristic of the structure sheaf of $X$. We have
\begin{align*}
\beta &=h^0(X, \mathcal{O}_X)-h^1(X, \mathcal{O}_X)+h^2(X, \mathcal{O}_X)
\\&=h^0(X, \mathcal{O}_X)-h^1(X, \omega_X)+h^0(X, \omega_X)
\\&=1-h^1(Z,\mathcal{O}_Z(-2+l,0))+h^0(Z,\mathcal{O}_Z(-2+l,0))
\\&=1+(-2+l+1)=l.
\end{align*}
This matches with the result that $\omega_X=\pi^*(\mathcal{O}_{\mathbb{P}^1}(\beta-2))$ in Lemma \ref{Kx}.

Let $A$ be the ample line bundle $\mathcal{O}_X(m,1)=\eta^*p_1^*\mathcal{O}_{\mathbf{P}_1}(m)\otimes \eta^*p_2^*\mathcal{O}_{\mathbf{P}_2}(1)$. The following lemma computes the intersection number of the ample line bundle with a fiber of $\pi$.

\begin{lemma}\label{A.E=2}
Let $E$ be any fiber of $\pi:X\rightarrow \mathbf{P}_1=\mathbb{P}^1$. The intersection number $A\cdot E=2$ and $A^2=4m$, where $A=\mathcal{O}_X(m,1)$.
\end{lemma}
\begin{proof}
Let $D_1, D_2\in H^2(\mathbf{P}_1\times \mathbf{P}_2, \mathbb{Z})$ be the first Chern class of $p_1^*(\mathcal{O}_{\mathbf{P}_1}(1))$ and $P_2^*(\mathcal{O}_{\mathbf{P}_2}(1))$ respectively. Then the cohomology ring $ H^*(\mathbf{P}_1\times \mathbf{P}_2, \mathbb{Z})$ is generated by $D_1, D_2$ with relations $D_1^2=0$, $D_2^2=0$ and $D_1\cdot D_2$ equals the class of a point $[\cdot]\in H^4(\mathbf{P}_1\times \mathbf{P}_2, \mathbb{Z})$. Let $\overline{D}_1=\eta^*D_1$ and $\overline{D}_2=\eta^*D_2$. Then in the cohomology ring $ H^*(X, \mathbb{Z})$, we have the relations $\overline{D}_1^2=0$, $\overline{D}_2^2=0$ and $\overline{D}_1\cdot \overline{D}_2=2[\cdot]$ since $\eta$ is a double cover, where $[\cdot]$ is the class of a point in $H^4(X, \mathbb{Z})$.

\smallskip

The class corresponding to divisor $A$ in $ H^2(X, \mathbb{Z})$ is $[A]=m\overline{D}_1+\overline{D}_2$ and
the class corresponding to fiber $E$ in $ H^2(X, \mathbb{Z})$ is $[E]=\overline{D}_1$. Thus
\begin{align*}
&(m\overline{D}_1+\overline{D}_2)\cdot \overline{D}_1=2[\cdot],
\\&(m\overline{D}_1+\overline{D}_2)^2=2m\overline{D}_1\cdot \overline{D}_2=4m[\cdot].
\end{align*}
Therefore the intersection number $A\cdot E=2$ and $A^2=4m$.
\end{proof}

 Let $A=\mathcal{O}_X(m,1)$, if $m\geq 3\beta$, then $A^2=4m\geq4(3\beta)>10\beta$. By Remark \ref{Euleromega1}, it is possible that Bott vanishing holds for $(X, A)$. We will give a criterion for when Bott vanishing holds for $(X, A)$ if $m$ is big enough.

\begin{proposition}\label{H1=0r=2}
With the same notations mentioned above,
let $X$ be the elliptic surface with a double cover $\eta:X\rightarrow Z=\mathbf{P}_1\times \mathbf{P}_2$ ramified over a smooth curve of bidegree $(2l,4)$ in $Z$ for some positive integer $l$ and $A$ be the ample line bundle $\mathcal{O}_X(m,1)$ on $X$. Assume and $m\geq11\beta-1$. Then $H^1(X,\Omega_{X}^1\otimes A)= 0$ if and only if $\pi$ has no fibers of type $\mathrm{III}$.
\end{proposition}
\begin{proof}
Since $m\geq11\beta-1$, the line bundle $A-\beta E$ equals $\mathcal{O}_X(m-\beta,1)=\eta^*p_1^*\mathcal{O}_{\mathbf{P}_1}(m-1)\otimes \eta^*p_2^*\mathcal{O}_{\mathbf{P}_2}(1)$ is ample, thus is nef and big. Let $L=A-(11\beta-1)E=\mathcal{O}_X(m-(11\beta-1),1).$ By Equation \ref{tensorL}, we have $$H^0(L)=H^0(Z, \mathcal{O}_Z(m-(11\beta-1),1))\oplus H^0(Z, \mathcal{O}_Z(m-(11\beta-1)-l,-1)).$$ Then $h^0=2(m-(11\beta-1)+1)$ by K\"unneth formula. Since $m-(11\beta-1)-1\geq -1$, we have $h^0(\mathbb{P}^1, \mathcal{O}(m-(11\beta-1)-1))=m-(11\beta-1)$. Similarly, by K\"unneth formula and Equation \ref{tensorL}, we obtain $h^0(L-E)=2(m-(11\beta-1))$. Thus $h^0(L)-h^0(L-E)=2$ which equals $A\cdot E$ by Lemma \ref{A.E=2}.
By Theorem \ref{Thr>1}, the cohomology space $H^1(X,\Omega_{X}^1\otimes A)= 0$ if and only if $\pi$ has no fibers of type $\mathrm{III}$.
\end{proof}

\begin{proposition}
With the same notations mentioned above,
let $X$ be the elliptic surface with a double cover $\eta:X\rightarrow Z=\mathbf{P}_1\times \mathbf{P}_2$ ramified over a smooth curve of bidegree $(2l,4)$ in $Z$ for some positive integer $l$ and $A$ be the ample line bundle $\mathcal{O}_X(m,1)$ on $X$. Assume $m\geq11\beta-1$. Then $(X, A)$ satisfies Bott vanishing if and only if $\pi$ has no fibers of type $\mathrm{III}$.
\end{proposition}
\begin{proof}


Since $\beta=l\geq1$, the assumption $m\geq11\beta-1$ implies $m-\beta+2>0$. Thus $A-K_X=A-(\beta-2)E=\mathcal{O}_X(m-\beta+2,1)$ is ample, so is nef and big.
Then the conclusion follows from Proposition \ref{pairbott} and Proposition \ref{H1=0r=2}.
\end{proof}

\subsection{$A\cdot E=3$}\label{e.g.r=3}
{\bf Elliptic fibrations with $r=A\cdot E=3$ for an ample line bundle $A$ on the elliptic surfaces, where $E$ is a fiber.}

Let $X$ be a smooth hypersurface $X_{a,3}\subset Z=\mathbb{P}^1\times \mathbb{P}^2$ of bidegree $(a,3)$, where $a$ is a positive integer. Let $p,q$ be the projection from $Z$ to $\mathbb{P}^1$ and $\mathbb{P}^2$ respectively. Let $\pi$ be the restriction of $p$ to $X_{a,3}$. A general fiber of $\pi$ is an elliptic curve since it is a degree $3$ hypersurface in $\mathbb{P}^2$. Thus $\pi:X=X_{a,3}\rightarrow \mathbb{P}^1$ is an elliptic fibration. The canonical line bundle $\omega_X$ equals $$(\omega_Z\otimes\mathcal{O}(X))|_X=\big(\mathcal{O}_Z(-2,-3)\otimes \mathcal{O}_Z(a,3)\big)|_X=\mathcal{O}_X(a-2,0)=\pi^*(\mathcal{O}_{\mathbb{P}^1}(a-2)).$$
If $a=1$, the dimension of $H^0(X, \mathcal{O}(nK_X))=H^0(\mathbb{P}^1, \mathcal{O}_{\mathbb{P}^1}(n(a-2)) )$ is $0$ for all $n\geq 1$ \cite[Example $\mathrm{VII}.3$]{A.Beau}. Hence $X$ has Kodaira dimension $\kappa=-\infty$. If $a=2$, the canonical bundle $\omega_X$ is trivial. Hence $X$ is a $K3$ elliptic surface with $\kappa=0$. The Bott vanishing for $K3$ elliptic surfaces is considered in \cite{BTbott}. In this paper, we are more interested in the case that $a\geq3$ in which $X$ has Kodaira dimension $\kappa=1$.

\smallskip

Let $\beta=\chi(\mathcal{O}_X)$ be the Euler characteristic of the structure sheaf of $X$. We have
\begin{align*}
\beta &=h^0(X, \mathcal{O}_X)-h^1(X, \mathcal{O}_X)+h^2(X, \mathcal{O}_X)
\\&=h^0(X, \mathcal{O}_X)-h^1(X, \omega_X)+h^0(X, \omega_X)
\\&=1-h^1(\mathbb{P}^1,\mathcal{O}_{\mathbb{P}^1}(a-2))+h^0(\mathbb{P}^1,\mathcal{O}_{\mathbb{P}^1}(a-2))
\\&=1-0+a-1=a.
\end{align*}
This matches with the result that $\omega_X=\pi^*(\mathcal{O}_{\mathbb{P}^1}(\beta-2))$ in Lemma \ref{Kx}.

\smallskip

Let $A$ be the ample line bundle $\mathcal{O}_X(m,1)=i^*p^*\mathcal{O}_{\mathbb{P}^1}(m)\otimes i^*q^*\mathcal{O}_{\mathbb{P}^2}(1)$, where $m\geq1$ and $i$ is the embedding from $X=X_{a,3}$ to $Z=\mathbb{P}^1\times \mathbb{P}^2$. The following lemma computes the intersection number of the ample line bundle with a fiber of $\pi$.

\begin{lemma}\label{A.E=3}
Let $E$ be any fiber of $\pi:X\rightarrow \mathbb{P}^1$. The intersection number $A\cdot E=3$ and $A^2=6+a$, where $A=\mathcal{O}_X(m,1)$.
\end{lemma}
\begin{proof}
Let $D, B\in H^2(\mathbb{P}^1\times \mathbb{P}^2, \mathbb{Z})$ be the first Chern class of $p^*(\mathcal{O}_{\mathbb{P}^1}(1))$ and $q^*(\mathcal{O}_{\mathbb{P}^2}(1))$ respectively. Then the cohomology ring $ H^*(\mathbb{P}^1\times \mathbb{P}^2, \mathbb{Z})$ is generated by $D, B$ with relations $D^2=0$, $B^3=0$ and $D\cdot B^2$ equals the class of a point $[\cdot]\in H^6(\mathbb{P}^1\times \mathbb{P}^2, \mathbb{Z})$. The class corresponding to $X$ in $H^2(\mathbb{P}^1\times \mathbb{P}^2, \mathbb{Z})$ is $[X]=aD+3B$. Then $$(D\cdot B)|_X=D\cdot B\cdot [X]=D\cdot B\cdot (aD+3B)=3[\cdot].$$ The class corresponding to fiber $E$ in $ H^2(X, \mathbb{Z})$ is $[E]=D|_X$.
Also the Chern class of the line bundle $A$ in $H^2(X, \mathbb{Z})$ equals $(mD+B)|_X$. Thus
\begin{align*}
&((mD+B)\cdot D)|_X=(mD+B)\cdot D[X]=(mD+B)D(aD+3B)=3[\cdot],
\\&(A\cdot A)|_X=A^2\cdot[X]=(mD+B)^2\cdot(aD+3B)=(6m+a)[\cdot].
\end{align*}
Therefore the intersection number $A\cdot E=3$ and $A^2=6m+a$.
\end{proof}

Let $A=\mathcal{O}_X(m,1)$. If $m\geq \frac{3\beta}{2}$, then $A^2=6m+a\geq6\cdot\frac{3\beta}{2}+\beta=10\beta$. By Remark \ref{Euleromega1}, it is possible that Bott vanishing holds for $(X, A)$. We will give a criterion for when Bott vanishing holds for $(X, A)$ if $m$ is big enough.

\begin{proposition}\label{H1=0r=3}
Let $X$ be the elliptic surface $X_{a,3}\subset Z=\mathbb{P}^1\times \mathbb{P}^2$ with elliptic fibration $\pi:X_{a,3}\rightarrow \mathbb{P}^1$ and $A$ be the ample line bundle $\mathcal{O}_X(m,1)$ on $X$, where $a\geq1$ and $m\geq11\beta-1$. Then $H^1(X,\Omega_{X}^1\otimes A)= 0$ if and only if $\pi$ has no fibers of type $\mathrm{IV}$.
\end{proposition}
\begin{proof}
A fiber of $\pi$ corresponds to the line bundle $E=\pi^*(\mathcal{O}_{\mathbb{P}^1}(m))$. We consider the line bundle $L=A-(11\beta-1)E=\mathcal{O}_X(m-(11\beta-1),1).$
If $m>11\beta-1$, then $L$ is ample, thus is nef and big. If $m=11\beta-1$, then $L=\mathcal{O}_X(0,1)=i^*q^*\mathcal{O}_{\mathbb{P}^2}(1)$ is nef. By Lemma \ref{A.E=3}, we get $$L^2=A^2-6(11\beta-1)=6m+a-6(11\beta-1)=a\geq1.$$ Thus $L$ is nef and big. The line bundle $A-\beta E=L+(11\beta-2)E$ is also nef and big.

\smallskip

Now we compute $h^0(L)-h^0(L-E)$.
Closed embedding of $X\hookrightarrow Z$ induces the exact sequence
\begin{equation}\label{OX}
0\rightarrow\mathcal{O}(-a,-3)\rightarrow\mathcal{O}_Z\rightarrow \mathcal{O}_X\rightarrow 0.
\end{equation}
Since $\beta=a$, we tensor this short exact sequence by $L$ to get
$$0\rightarrow\mathcal{O}_{Z}(m-(12\beta-1),-2)\rightarrow\mathcal{O}_{Z}(m-(11\beta-1),1)\rightarrow \mathcal{O}_{X}(m-(11\beta-1),1)\rightarrow 0$$ and consider the induced long exact sequence.
Since $H^i(\mathbb{P}^2, \mathcal{O}_{\mathbb{P}^2}(-2))=0$ for $i\geq 0$, we have $H^i(Z, \mathcal{O}_{Z}(-11\beta+2,1))=0$ for $i=0,1$ by K\"unneth formula. Thus $H^0(X, \mathcal{O}_{X}(m-(11\beta-1),1))$ is isomorphic to $$H^0(Z, \mathcal{O}_{Z}(m-(11\beta-1),1))=H^0(\mathbb{P}^1, \mathcal{O}_{\mathbb{P}^1}(m-(11\beta-1)) )\otimes H^0(\mathbb{P}^2, \mathcal{O}_{\mathbb{P}^2}(2)).$$ Hence $h^0(L)=3(m-(11\beta-1)+1)$. Tensoring the short exact sequence (\ref{OX}) by $L-E=\mathcal{O}_X(m-(11\beta-1)-1,1)$, we get $h^0(L)=3(m-11\beta+1)$ with the similar arguments. Therefore $h^0(L)-h^0(L-E)=3$.

\smallskip

By Lemma \ref{A.E=3}, the intersection number $A\cdot E=3$. Then Theorem \ref{Thr>1} shows that $H^1(X,\Omega_{X}^1\otimes A)\neq 0$ if and only if $\pi$ has a fiber of type $\mathrm{IV}$. This implies the conclusion.
\end{proof}

\begin{proposition}\label{r=3bott}
Let $X$ be the elliptic surface $X_{a,3}\subset Z=\mathbb{P}^1\times \mathbb{P}^2$ with elliptic fibration $\pi:X_{a,3}\rightarrow \mathbb{P}^1$ and $A$ be the ample line bundle $\mathcal{O}_X(m,1)$ on $X$, where $a\geq1$ and $m\geq11\beta-1$. Then $(X, A)$ satisfies Bott vanishing if and only if $\pi$ has no fibers of type $\mathrm{IV}$.
\end{proposition}
\begin{proof}


Since $\beta=a\geq 1$, then $m\geq11\beta-1$ imples $m-\beta+2>0$. Hence $A-K_X=A-(\beta-2)E=\mathcal{O}_X(m-\beta+2,1)$ is ample, so is nef and big.
The conclusion follows from Proposition \ref{pairbott} and Proposition \ref{H1=0r=3}.
\end{proof}

\subsection{$A\cdot E=4$}\label{e.g.r=4}
{\bf Elliptic fibrations with $r=A\cdot E=4$ for an ample line bundle $A$ on the elliptic surfaces, where $E$ is a fiber.}

\smallskip

Let surface $X$ be the complete intersection of two hypersurfaces $X_{a, 2}$ and $X_{b, 2}$ in $Z=\mathbb{P}^1\times \mathbb{P}^3$, where $a$ and $b$ are positive integers.
Let $p,q$ be the projection from $Z$ to $\mathbb{P}^1$ and $\mathbb{P}^3$ respectively. Let $\pi$ be the restriction of $p$ to $X$. Every fiber $E$ of $\pi$ is a complete intersection of two hypersurfaces $S_1$ and $S_2$ in $\mathbb{P}^3$ of degree $2$.  The canonical bundle of $S_1$ is $$\omega_{S_1}=(\omega_{\mathbb{P}^3}\otimes\mathcal{O}(S_1))|_{S_1}=(\mathcal{O}_{\mathbb{P}^3}(-4)\otimes\mathcal{O}_{\mathbb{P}^3}(2))|_{S_1}=\mathcal{O}_{S_1}(-2).$$ Then the canonical bundle of $E$ is trivial since $$\omega_E=(\omega_{S_1}\otimes\mathcal{O}(E))|_E=\big(\mathcal{O}_{S_1}(-2)\otimes\mathcal{O}_{S_1}(2)\big)|_E=\mathcal{O}_E.$$ Thus $E$ is an elliptic curve.
Therefore $\pi:X\rightarrow \mathbb{P}^1$ is an elliptic fibration. The canonical line bundle $\omega_{X_{a,2}}$ equals $$(\omega_Z\otimes\mathcal{O}(X_{a,2}))|_{X_{a,2}}=\big(\mathcal{O}_Z(-2,-4)\otimes \mathcal{O}_Z(a,2)\big)|_{X_{a,2}}=\mathcal{O}_{X_{a,2}}(a-2,-2).$$ Then the canonical line bundle of $X$ is $$\omega_X=\big(\omega_{X_{a,2}}\otimes \mathcal{O}(X)\big)|_{X}=\big(\mathcal{O}_{X_{a,2}}(a-2,-2)\otimes \mathcal{O}_{X_{a,2}}(b,2)\big)|_{X}=\mathcal{O}_X(a+b-2,0)$$ which equals $\pi^*(\mathcal{O}_{\mathbb{P}^1}(a+b-2))$.

\smallskip

If $a+b=2$, the canonical bundle $\omega_X$ is trivial. Hence $X$ is a $K3$ elliptic surface with $\kappa=0$. The Bott vanishing for $K3$ elliptic surfaces is considered in \cite{BTbott}.
In this paper, we are more interested in the case that $a+b\geq3$ in which $X$ has Kodaira dimension $\kappa=1$.
Let $\beta=\chi(\mathcal{O}_X)$ be the Euler characteristic of the structure sheaf of $X$.
With the similar computation in Section \ref{e.g.r=3},  we have $\beta=a+b$.
This matches with the result that $\omega_X=\pi^*(\mathcal{O}_{\mathbb{P}^1}(\beta-2))$ in Lemma \ref{Kx}.

\smallskip

Let $A$ be the ample line bundle $\mathcal{O}_X(m,1)=i^*p^*\mathcal{O}_{\mathbb{P}^1}(m)\otimes i^*q^*\mathcal{O}_{\mathbb{P}^3}(1)$, where $i$ is the embedding from $X$ to $Z=\mathbb{P}^1\times \mathbb{P}^3$. The following lemma computes the intersection number of the ample line bundle with a fiber of $\pi$.

\begin{lemma}\label{A.E=4}
Let $E$ be any fiber of $\pi:X\rightarrow \mathbb{P}^1$. The intersection number $A\cdot E=4$ and $A^2=2a+2b+8$, where $A=\mathcal{O}_X(m,1)$.
\end{lemma}
\begin{proof}
Let $D, B\in H^2(\mathbb{P}^1\times \mathbb{P}^3, \mathbb{Z})$ be the first Chern class of $p^*(\mathcal{O}_{\mathbb{P}^1}(1))$ and $q^*(\mathcal{O}_{\mathbb{P}^3}(1))$ respectively. Then the cohomology ring $ H^*(\mathbb{P}^1\times \mathbb{P}^2, \mathbb{Z})$ is generated by $D, B$ with relations $D^2=0$, $B^4=0$ and $D\cdot B^3$ equals the class of a point $[\cdot]\in H^8(\mathbb{P}^1\times \mathbb{P}^3, \mathbb{Z})$. The class corresponding to $X$ in $H^4(\mathbb{P}^1\times \mathbb{P}^3, \mathbb{Z})$ is $[X]=[X_{a,2}]\cdot[X_{a,2}]=(aD+2B)\cdot(bD+2B)$. Then $$(D\cdot B)|_X=D\cdot B\cdot [X]=D\cdot B\cdot (aD+2B)\cdot(bD+2B)=4[\cdot].$$ The class corresponding to fiber $E$ in $ H^2(X, \mathbb{Z})$ is $[E]=D|_X$.
Also the Chern class of the line bundle $A$ in $H^2(X, \mathbb{Z})$ equals $(mD+B)|_X$. Thus
\begin{align*}
&((mD+B)\cdot D)|_X=(mD+B)\cdot D\cdot[X]=B\cdot D[X]=4[\cdot],
\\&(A\cdot A)|_X= (mD+B)^2|_X=2m(B\cdot D)|_X+B^2|_X=8m[\cdot]+ B^2|_X,
\\& B^2|_X=B^2\cdot[X]=B^2\cdot(aD+2B)\cdot(bD+2B))=(2a+2b)[\cdot].
\end{align*}
Therefore the intersection number $A\cdot E=4$ and $A^2=2a+2b+8m$.
\end{proof}

Let $A=\mathcal{O}_X(m,1)$. If $m\geq a+b$, then $A^2\geq 10(a+b)=10\beta$. Then by Remark \ref{Euleromega1}, it is possible that Bott vanishing holds for $(X, A)$. We will show Bott vanishing holds for $(X, A)$ if $m$ is big enough.

\begin{proposition}\label{H1=0r=4}
Let $X$ be an elliptic surface constructed as the complete intersection of two hypersurfaces $X_{a, 2}$ and $X_{b, 2}$ in $Z=\mathbb{P}^1\times \mathbb{P}^3$ and $A$ be the ample line bundle $\mathcal{O}_X(m,1)$ on $X$,, where $a$ and $b$ are positive integers. If $m> 12(a+b)-2$, then $H^1(X,\Omega_{X}^1\otimes A)$ is zero.
\end{proposition}
\begin{proof}
Since $m> 12(a+b)-2$ and $\beta=a+b$, the line bundle $A-(12\beta-2)=\mathcal{O}_X(m-(12\beta-2),1)=i^*p^*\mathcal{O}_{\mathbb{P}^1}(m-(12\beta-2))\otimes i^*q^*\mathcal{O}_{\mathbb{P}^2}(1)$ is ample, thus is nef and big. Since $\beta=a+b\geq1$, the cohomology space $H^1(X,\Omega_{X}^1\otimes A)=0$ by Corollary \ref{cor>1}.
\end{proof}

\begin{proposition}\label{r=4bott}
Let $X$ be an elliptic surface constructed as the complete intersection of two hypersurfaces $X_{a, 2}$ and $X_{b, 2}$ in $Z=\mathbb{P}^1\times \mathbb{P}^3$ and $A$ be the ample line bundle $\mathcal{O}_X(m,1)$ on $X$,, where $a$ and $b$ are positive integers. If $m> 12(a+b)-2$, then $(X, A)$ satisfies Bott vanishing.
\end{proposition}
\begin{proof}


Since $\beta=a+b\geq 1$, then $m>12(a+b)-2=12\beta-2$ imples $m-\beta+2>0$. Hence $A-K_X=A-(\beta-2)E=\mathcal{O}_X(m-\beta+2,1)$ is ample, so is nef and big. 
The conclusion follows from Proposition \ref{pairbott} and Proposition \ref{H1=0r=4}.
\end{proof}

\end{document}